\documentclass[11pt]{amsart}
\usepackage{amsmath}
\usepackage{amssymb}
\usepackage{commath}
\usepackage{mathtools}
\usepackage{comment}
\usepackage{mathrsfs}
\usepackage{hyperref}
\usepackage{xcolor}
\usepackage{derivative}
\usepackage{inputenc}
\usepackage{graphicx}
\usepackage{derivative}
\usepackage{marginnote}
\usepackage{stackengine} 
\usepackage{fullpage}

\usepackage{tabularx}
\usepackage{ltablex}

\theoremstyle{plain}
\newtheorem{theorem}{Theorem}
\newtheorem{lemma}[theorem]{Lemma}
\newtheorem{proposition}[theorem]{Proposition}
\newtheorem{corollary}[theorem]{Corollary}

\theoremstyle{definition}
\newtheorem{definition}[theorem]{Definition}
\newtheorem{example}[theorem]{Example}
\newtheorem{remark}[theorem]{Remark}

\numberwithin{theorem}{section}
\numberwithin{equation}{section}

\newcommand{\B}{\mathbb{B}}
\newcommand{\C}{\mathbb{C}}
\newcommand{\N}{\mathbb{N}}

\newcommand{\R}{\mathbb{R}}

\newcommand{\G}{G}

\newcommand{\SU}{\mathrm{SU}}
\newcommand{\U}{\mathrm{U}}

\newcommand{\tr}{\mathrm{tr}}

\newcommand{\cH}{\mathcal{H}}

\newcommand{\cA}{\mathcal{A}}
\newcommand{\cO}{\mathcal{O}}

\newcommand{\cS}{\mathcal{S}}
\newcommand{\cF}{\mathcal{F}}
\newcommand{\cB}{\mathcal{B}}
\newcommand{\cK}{\mathcal{K}}
\newcommand{\cC}{\mathcal{C}}
\newcommand{\cL}{\mathcal{L}}
\newcommand{\cX}{\mathcal{X}}

\newcommand{\ip}[2]{\langle #1,#2 \rangle}

\newcommand{\vm}{\textbf{m}}

\newcommand{\Berg}{\cA^2(\B^n)} 
\newcommand{\Bergalph}{\cA^2_\alpha(\B^n)}

\newcommand{\cbufL}{C_{b,u}^{(L)}(\B^n)} 
\newcommand{\cbufR}{C_{b,u}^{(R)}(\B^n)} 

\newcommand{\bddf}{L^\infty(\B^n)} 
\newcommand{\bddfG}{L^\infty(G)} 
\newcommand{\Lone}{L^1(\B^n,d\lambda)} 
\newcommand{\Lp}{L^p(\B^n,d\lambda)} 
\newcommand{\Lq}{L^{p'}(\B^n,d\lambda)} 
\newcommand{\LoneG}{L^1(\G)} 

\newcommand{\LtwoG}{L^2(G,d\mu_0)}
\newcommand{\LpG}{L^p(G,d\mu_G)}

\newcommand{\cbuopL}{C_{b,u}^{(L)}(\cA^2)}

\newcommand{\cbuopLrad}{C_{b,u}^{(L)}(\cA^2)^{K}}

\newcommand{\bdd}{\cB(\cA^2)}

\newcommand{\bddLtwo}{\cB(\LtwoG)}
\newcommand{\bddH}{\cB(\cH)}

\newcommand{\traceop}{\cS^1(\cA^2)}
\newcommand{\schop}{\cS^p(\cA^2)}
\newcommand{\schoprad}{\cS^p(\cA^2)^{K}}
\newcommand{\schopq}[1]{\cS^{#1}(\cA^2)}

\newcommand{\toepalg}{\mathfrak{T}(L^\infty)}
\newcommand{\toepalgrad}{\mathfrak{T}(L^\infty)^{K}}

\newcommand{\rad}{\bdd^{K}}

\newcommand{\haar}[1]{\ d\mu_\G(#1)}
\newcommand{\haarK}[1]{\ d\mu_K(#1)}
\newcommand{\inv}[1]{\ d\lambda(#1)}

\newcommand{\Act}[1]{\tau_{{#1}}}
\newcommand{\act}[2]{\tau_{{#1}}({#2})}

\newcommand{\actfL}[2]{\ell_{#1}#2}
\newcommand{\actopL}[2]{(\widetilde{\pi}({#1})#2)} 
\newcommand{\actopLnoouterparens}[2]{\widetilde{\pi}({#1})#2} 
\newcommand{\actftL}[2]{\ell_{#1}(#2)}
\newcommand{\actoptL}[2]{\widetilde{\pi}_{#1}(#2)}

\newcommand{\actfR}[2]{r_{#1}(#2)}

\newcommand{\actftR}[2]{r_{#1}(#2)}

\newcommand{\piL}{\pi}

\newcommand{\Gast}{\ast_\pi}
\newcommand{\Rast}{\ast_r}
\newcommand{\Last}{\ast_\pi}
\newcommand{\last}{\ast_\ell}

\newcommand{\Tr}{\operatorname{Tr}}

\newcommand{\Ber}{\widetilde{B}}

\newcommand{\Rad}[1]{{#1}^{\#}}

\newcommand{\one}{\boldsymbol{1}}

\newcommand{\op}{\infty}
\newcommand{\Id}{\mathrm{Id}}

\newcommand{\Keralph}[2]{K^\alpha_{#1}(#2)}
\newcommand{\keralph}[2]{k^\alpha_{#1}(#2)}

\newcommand{\converges}[1]{\stackrel{#1\rightarrow\infty}{\longrightarrow}}

\DeclareMathOperator{\diag}{diag}

\newcommand{\tens}[2]{ #1 \otimes \overline{#2} }

\newcommand{\blue}[1]{{\color{blue} #1}} 
\newcommand{\red}[1]{{\color{red} #1}}

\newcommand{\dv}{dv}
\newcommand{\dvz}{dv (z)}
\newcommand{\actL}{\ell}
\newcommand{\actR}{r}

\newcommand{\rI}{{\mathrm I}}
\newcommand{\tpi}{\widetilde{\pi}}
\newcommand{\ds}{\displaystyle}
\newcommand{\wpi}{\widetilde{\pi}}
\newcommand{\bm}{\mathbf{m}}

\title[QHA in Bergman Spaces]{Quantum Harmonic Analysis on the Unweighted Bergman Space of the Unit Ball}

\begin{document}

\author{Matthew Dawson}
\address{CONAHCYT---CIMAT Unidad Mérida,
          Parque Científico y Tecnológico de Yucatán,
          KM 5.5 Carretera Sierra Papacal --- Chuburná Puerto,
          Sierra Papacal; Mérida, YUC 97302, México}
\email{matthew.dawson@cimat.mx}

\author{Vishwa Dewage}
\address{Department of Mathematics, Clemson University, South Carolina, SC 29635}
\email{vdewage@clemson.edu}

\author{Mishko Mitkovski}
\address{Department of Mathematics, Clemson University, South Carolina, SC 29635}
\email{mmitkov@clemson.edu}

\author{Gestur {\'O}lafsson}
\address{Department of Mathematics, Louisiana State University, Baton Rouge, LA 70803, USA.} 
\email{olafsson@math.lsu.edu}
\thanks{M.~M. was supported by NSF grant DMS-2000236.}
\thanks{G.~\'O. was supported by Simons grant 58606.}

\begin{abstract}
    We study quantum harmonic analysis (QHA) on the Bergman space $\mathcal{A}^2(\mathbb{B}^n)$ over the unit ball in $\mathbb{C}^n$. We formulate a Wiener's Tauberian theorem, and characterizations of the radial Toeplitz algebra over $\mathcal{A}^2(\mathbb{B}^n)$.
    We discuss the $\alpha$-Berezin transform and investigate the question of approximations by Toeplitz operators.
\end{abstract}


\maketitle

\setcounter{tocdepth}{1}
\tableofcontents

\section{Introduction}
With the recent interest in R. Werner's Quantum Harmonic Analysis (QHA), we now have a new set of tools that are applicable to several areas in analysis \cite{BBLS22, FH23, F19, FR23, KLSW12, LS20}. As for operator theory, QHA has not only clarified proofs of well-known theorems but also brought to light intriguing new results, offering a fresh perspective. In \cite{F19}, Fulsche uses QHA to provide several meaningful characterizations of the Toeplitz algebra over the Fock space. By now, it is evident that QHA is well suited to the study of Toeplitz operators on the Fock space over $\C^n$ \cite{DM23, DM24, FH23, F19, FR23}. Our main goal is to understand the Toeplitz operators on the Bergman space over the unit ball $\B^n$ in $\C^n$ using QHA notions. 

We formulate  QHA on the Bergman space $\Berg$ over the unit ball $\B^n$ in $\C^n$, utilizing 
the scalar-type holomorphic discrete series of square-integrable representations of the noncommutative group $\SU(n,1)$ acting on the weighted Bergman spaces $\cA^2_\alpha(\B^n)$, where $\alpha>-1$.

As observed in \cite{KLSW12,LS18, W84}, Wiener's Tauberian theorem plays a fundamental role in QHA, which we discuss in Theorem~\ref{theo:WienerT for GK} in the context of radial operators. Then, by exploiting the fact that unit ball is a commutative space, we note that Fulsche's characterizations of the Toeplitz algebra for the Fock space holds for the radial Toeplitz algebra over the Bergman space $\Berg$. 
More importantly, using QHA, we investigate, for each $\alpha > -1$, an $\alpha$-Berezin transform \[
    \Ber_\alpha: \cB(\Berg) \rightarrow L^\infty(\B^n).
\]
This transform shares many of the interesting properties of the $\alpha$-Berezin transform introduced for integer values of $\alpha$ by Suarez in \cite{S04} for the unit disc $\mathbb{D}$, and was generalized to the unit ball in \cite{BHV14}; for instance one can characterize the Toeplitz operators using the this $\alpha$-Berezin transform as we note in Theorem~\ref{theo:Toeplitz criterion}, as was first proven for the Bergman space over the unit disc $\cA^2(\mathbb{D})$ in \cite{S04}. However, we emphasize that using QHA one can define the $\alpha$-Berezin transform with a clear motivation and this leads to much simpler proofs. We also extend this transform for non-integer values of $\alpha$.
The main focus of this paper is to explore the question of approximating a bounded operator $S\in \cB(\Berg)$ by Toeplitz operators of the form $T_{\Ber_\alpha(S)}$ as $\alpha\rightarrow\infty$. It is known that this convergence holds for radial operators in the Toeplitz algebra \cite{BHV14b,S05} and also for any Toeplitz operator \cite{S04,S07}. We find a larger class of operators in the Toeplitz algebra $\toepalg$ for which $S=\lim_{\alpha\rightarrow \infty} T_{\Ber_\alpha(S)}$ in the operator norm. We also discuss classes of operators for which this convergence holds in the Schatten-$p$ norm as well as in the strong operator topology. 

The article is organized as follows. In Section~\ref{sec:prelim}, we collect some facts about Toeplitz operators, convolutions and group representations on Bergman spaces.  In Section~\ref{sec:QHA Bergman}, we discuss QHA for the Bergman space. Our main results are contained in Sections~\ref{sec:Wiener T} and \ref{sec:Berezin transform}. In Section~\ref{sec:Wiener T}, we prove a version of Wiener's Tauberian theorem for operators on $\Berg$. In Section \ref{sec:Berezin transform}, we discuss the $\alpha$-Berezin transform, discuss the convergence $T_{\Ber_\alpha(S)} \stackrel{\alpha\rightarrow\infty}{\longrightarrow} S$ and show that the radial Toeplitz algebra coincides with the algebra of left-uniformly-continuous radial operators. In Appendix ~\ref{sec:general QHA}, we provide an overview of QHA for locally compact groups as discussed in \cite{H23}, for the benefit of the reader unfamiliar with QHA, while also including some convergence properties that we need. In Appendix \ref{sec:table}, we include a table of notation.

\section{Preliminaries}\label{sec:prelim}
Let $n\in\N$ and let $\alpha>-1$.  The \textit{weighted Bergman space} $\Bergalph$ over the unit ball $\B^n$ in $\C^n$ is 
the reproducing-kernel Hilbert  space of all holomorphic functions on the ball that are square-integrable with respect to the measure
\[
   d\mu_\alpha(z)=C_\alpha (1-|z|^2)^{\alpha}\dvz
\]
where $C_\alpha=\frac{\Gamma (n+\alpha +1)}{n!\Gamma(\alpha+1)}=\frac{(n+\alpha)!}{n! \alpha !}$ and $\dv$ denotes the normalized Lebesgue measure of the ball. Here, as in most of this article, we follow the notation from \cite{Z05}.

The reproducing kernel $K^\alpha$ is given by
\begin{equation}
    \label{eq:reproducing_kernel_definition}
    K^\alpha(w,z)=\Keralph{z}{w}=\frac{1}{(1-\Bar{z}w)^{n+1+\alpha}},\ \ z,w\in \B^n.
\end{equation}
The Bergman projection $P:L^2(\B^n,\mu_\alpha)\to \Bergalph$ is given by
\[
    (Pf)(z)=\ip{f}{K^\alpha_z},\ \ z\in \B^n,\ \ f\in L^2(\B^n,\mu_\alpha).
\]

We will also use the  normalized reproducing kernel $k^\alpha_z$ defined by:
\begin{equation}\label{eq:kz}
   \keralph{z}{w}=\frac{\Keralph{z}{w}}{\|K^\alpha_z\|}=\frac{(1-|z|^2)^{(n+1+\alpha)/2}}{(1-\Bar{z}w)^{n+1+\alpha}},\ \ z,w\in \B^n.
\end{equation}
hen $\alpha=0$, we note that $\mu_0$ is the usual normalized Lebesgue measure on the unit ball, and we get the usual ``weightless'' Bergman space $\Berg$ and write $K$ and $k$ instead of $K_0$ and $k_0$, respectively. 

\begin{remark}While we need to utilize the parametrized Bergman spaces and associated notions for some calculations, our analysis will, for the sake of simplicity, primarily focus on the unweighted Bergman space $\Berg$, although we believe that every result in this article applies to $\Bergalph$ for all $\alpha > -1$.  In the forthcoming article \cite{DDO25}, we will carry out this analysis for all scalar-type discrete series representations on Bergman spaces for arbitrary unbounded symmetric domains.
\end{remark}

For a multi-index $\vm=(m_1,\ldots,m_n) \in \N_0^n$, we write $\vm ! =m_1!\cdots m_n!$ and similarly define
\begin{equation}\label{eq:onb}
  p_{\vm} (z)=z_1^{m_1}\cdots z_n^{m_n}\quad\text{and}\quad   e_{\vm}(z) =\sqrt{\frac{(n+|\vm|)!}{n!\vm!}} p_m (z),\quad z = (z_1,\cdots,z_n) \in \B^n .
  \end{equation}
Then $\{e_{\vm}\}_{\vm\in \N_0^n}$ is an orthonormal basis for $\Berg$. In particular, we have
\begin{equation}\label{eq:Konb}
K^\alpha (z,w) = \sum_{m\in\N_0^n}e_{\vm}(z)\overline{e_\vm (w)} .
\end{equation}

Denote by $\bdd$ the Banach space of all bounded operators on $\Berg$ with the usual operator norm $\|T\|=\sup_{\|v\|\le 1} \|T(v)\|$. Given a function $a:\B^n\to \C$, we define the \textit{Toeplitz operator $T_a$ with symbol $a$} by
\[
   T_af=P(af),\ \ f\in \Berg.
\]
If $a\in \bddf$ then $T_a\in \bdd$   with  $\|T_a\|\leq \|a\|_\infty$. We say that $a$ is the \textit{symbol} of $T_a$. Note that $T_a^* = T_{\overline a}$ and if $a$ is holomorphic and bounded then $T_a$ is just the
multiplication operator $f\mapsto af$. In general we do not have $T_aT_b = T_{ab}$. The \textit{Toeplitz algebra} $\toepalg$ 
is the $C^*$-algebra  generated by the Toeplitz operators $T_a$ with symbols $a\in \bddf$.

The Berezin transform of an operator $S\in \bdd$ is the function $B(S):\B^n \rightarrow \C$ given by
\begin{equation}\label{eq:BerS}
    B(S)(z)=\ip{Sk_z}{k_z},\ \ z\in \B^n.
\end{equation}
Recall that the Berezin transform of a function $f\in L^\infty(\B^n)$ is, by definition, the function $B(f):\B^n\rightarrow \C$ given by:
\begin{equation}\label{eq:BerFunction}
    B(f)(z):= B(T_f)(z) = \langle T_f k_z, k_z \rangle = \langle fk_z, k_z \rangle
\end{equation}
\subsection{Discrete Series Representations of $\SU(n,1)$ on the Bergman space}
Denote by $\SU (n,1)$ the group of linear transformations of $\C^{n+1}$ that preserve the hermitian form $\beta_{n,1}(z,w) = z_1\overline{w_1} + \cdots + z_n\overline{w_n} - z_{n+1}\overline{w_{n+1}}$ and have determinant one. In the following we write $\G=\SU(n,1)$ and  $K=\U(n)\subseteq \G$, where the embedding is given by \[A\mapsto k_A=\begin{pmatrix} A & 0\\ 0 & 1/\det A\end{pmatrix} .\] 
We write elements of $G$ as block matrices of the form $g(A,u,v,d)=\begin{pmatrix} A & u \\ v^* & d \end{pmatrix}$, where $A\in M(n,\C)$, $d\in \C$, and $u,v\in \C^n$, and where $v^*$ denotes the row matrix that is the adjoint of the column vector $v$. We usually write $g,x,y$, etc for elements in $\G$ and $k,h$ for elements in $K$.  It is well known that $G$ acts on $\B^n$ by Möbius transformations, namely:
\begin{equation}\label{def:Action}
    \left(\begin{matrix} A & u \\ v^* & d \end{matrix}\right) \cdot z := \frac{Az + u}{v^* z + d} = \frac{1}{\langle z, v\rangle +d} (Az+u), \quad z\in \B^n, \,\,  \begin{pmatrix} A & u \\ v^* & d \end{pmatrix}\in G .
\end{equation}
The action is transitive and the stabilizer of $0$ is $K$, which allows one to identify  $G/K$ with $\B^n $ by the map $gK \mapsto g\cdot 0$.

For all $\alpha > -1$ one defines the \textit{scalar-type holomorphic discrete series representation} or \textit{Bergman space representation} $\pi_\alpha:\G\to \mathrm{U}(\Bergalph)$ by
\begin{equation}
   \label{eq:representation_definition}
   (\pi_\alpha(g)f)(z)=j_\alpha(g^{-1},z)f(g^{-1}\cdot z),\quad f\in\Bergalph , 
\end{equation}
where the multiplier  $j_\alpha : \G \times \B^n\to \C$ is given by:
\begin{equation}
    \label{eqn:cocycle_def}
    j_\alpha( g, z) = (v^*z + d)^{-(\alpha+n+1)} = (\langle z, v \rangle+ d)^{-(\alpha + n + 1)},\quad
    z\in \B^n,\,\, g= \begin{pmatrix} A & u \\ v^* & d \end{pmatrix}  \in \G.
\end{equation}
For the case of $\alpha=0$, we write $\pi$ in place of $\pi_0$.   

If $\alpha>-1$ is not an integer, then $G$ must be replaced by a covering group so that a holomorphic branch of the formula in (\ref{eqn:cocycle_def}) exists, or else one views $\pi_\alpha$ as a projective representation of $\G$.  If $\alpha$ is rational, then this covering group can be chosen to be a finite covering of $G$ (and thus has finite center), but if $\alpha$ is irrational, then one must choose the full universal covering group $\widetilde{G}$.  In all cases $\pi_\alpha$ belongs to the relatively discrete series of $\widetilde{G}$.

We list here properties of $K^\alpha$ and the cocycle $j_\alpha$ that we will use several times in this article. Those are well known but we include the proofs for completeness. We note that (1) is equivalent to $\pi_\alpha$ being a unitary
representation. We also note that the idea in the proof of (1) is the same idea that leads to the action \eqref{def:Action}.

\begin{lemma}
    \label{lem:cocycle_normalized_ker}    The reproducing kernel $K_\alpha$ and the cocycle $j_\alpha$ satisfy the following identities:
    \begin{enumerate}
    \item $\displaystyle 
        j_\alpha(g,w)K_\alpha(g\cdot w,g\cdot z)\overline{j_\alpha(g,z)}= K_\alpha(w,z)$.
    \item $\displaystyle |j_\alpha(g,0)|^2=(1-|g\cdot 0|^2)^{n+1+\alpha}$.
    \item $\displaystyle j_\alpha (g^{-1},gz) = j(g,z)^{-1}\quad\text{and}\quad   \overline{j_\alpha(g^{-1},0)}=j_\alpha(g,0)$.
    \item $\displaystyle \overline{j_\alpha(g^{-1},h^{-1}\cdot 0)}=\frac{j_\alpha(h,g\cdot 0)j_\alpha(g,0)}{j_\alpha(h,0)}$.
    \item $\displaystyle j_\alpha(g^{-1},w)=\overline{j_\alpha(g,0)}K_\alpha(w,g\cdot 0)$.
    \end{enumerate}
\end{lemma}

\begin{proof} (1)  Let $g = g(A,u,v,d)\in \G$ and let $z\in\C^n$. Then
    \[
        g\begin{pmatrix}z\\ 1 \end{pmatrix} = \begin{pmatrix} Az + u\\ \ip{z}{v} + d\end{pmatrix}.
    \]
    Hence, as $g$ leaves $\beta_{n,1}$ invariant, we get, with $\gamma = -n -1 -\alpha$:
    \begin{align*}
    K^\alpha (z,w)&=(1-\ip{z}{w})^\gamma =(-\beta_{n,1}((z,1)^\top,(w,1)^\top))^\gamma\\
        & =((\ip{z}{v}+d)\overline{(\ip{w}{v}+d)}-\ip{Az+u}{Aw+u})^\gamma \\
        & =(\ip{z}{v}+d)^{\gamma}\overline{(\ip{w}{v}+d)}^{\gamma }(1- (\ip{z}{v}+d)^{-1}\overline{(\ip{w}{v}+d)}^{-1}\ip{Az+u}{Aw+u})^\gamma\\
        & = j _\alpha (g,z)\overline{j_\alpha (g,w)}K^\alpha (g\cdot z,g\cdot w).
    \end{align*}
        
    \noindent
    (2)  This follows from (1)  by taking $z=w=0$ and using that $K^\alpha (0,0)=1$. 
     \medskip
        
    \noindent
    (3) We  have $1=j_\alpha (g^{-1}g,z) = j_\alpha (g, z)j_\alpha(g^{-1},g\cdot z)$.  
    For the second part we note that $g\in \SU (n,1)$ implies that $g^*\rI_{n,1}g=\rI_{n,1}$ where $\rI_{n,1}=\diag (1,\ldots , 1, -1)$. Thus $g^{-1} = \rI_{n,1}g^* \rI_{n,1}$ which implies that $(g^{-1})_{n+1,n+1} = \overline{g_{n+1,n+1}}$. Now \eqref{eqn:cocycle_def} implies the claim because $j_\alpha (g,0) = g_{n+1,n+1}^{-\alpha - n -1}$.
     \medskip
        
    \noindent
    (4)  This follows from (3) by the following argument:
    \[
        \overline{j_\alpha (g^{-1},h^{-1}\cdot 0)j_\alpha (h^{-1},0)}
        = \overline{j_\alpha ((hg)^{-1},0)}
        = j_\alpha (hg,0)= j_\alpha (h,g\cdot 0)j_\alpha (g,0).
    \]
        
    \noindent
    (5) This follows
    by replacing $g\cdot w$ by $g^{-1}\cdot w$ and $g\cdot z$ by $0=g^{-1}g\cdot 0$ in (1) 
    and then using (3). 
\end{proof}

\begin{remark}\label{rem:palpha}
    As a consequence of Lemma~\ref{lem:cocycle_normalized_ker}, part 5, we can write the representation $\pi_\alpha$ as 
    \[
        (\pi_\alpha(g)f)(w)=\overline{j_\alpha(g,0)}K^\alpha_{g\cdot 0}(w) f(g^{-1}w),\ \ w\in \B^n.
    \]
    We also note that $j_\alpha(g^{-1},w)=\overline{j_\alpha(g,0)}K_\alpha(w,g\cdot 0)\neq |j_\alpha(g,0)|K_\alpha(w,g\cdot 0)=k^\alpha_{g\cdot 0}(w)$.
\end{remark}

    In what follows, we need the value of the formal dimension of the unweighted scalar-type holomorphic discrete series representation, $\pi := \pi_0$ (see Appendix~\ref{sec:general QHA} for the definition of the formal dimension of a square-integrable representation).  This is a straightforward calculation, first performed by Harish-Chandra~(see \cite{H56}), which we include for completeness and to guarantee consistency with our notation (in the representation-theory literature one often parametrizes the holomorphic discrete series with  $\lambda := \alpha + n +1$, so that $\lambda > n$) and choices of normalization of the measures.
\begin{lemma}
    For all $\alpha > -1$, the formal dimension of the representation $\pi_\alpha$ is given by $d_{\pi_\alpha}=C_\alpha$.  In particular, for the case of $\alpha=0$, one has that $d_\pi= 1$.
\end{lemma}
\begin{proof}
    Note that by Schur orthogonality relations 
    \begin{align*}
        \ip{{\pi_\alpha}_{1,1}}{{\pi_\alpha}_{1,1}}=\frac{1}{d_{\pi_\alpha}}\ip{1}{1}_\alpha\ip{1}{1}_\alpha
        = \frac{1}{d_{\pi_\alpha}}.
    \end{align*}
    By Lemma \ref{lem:cocycle_normalized_ker}, we have
    $$\pi_\alpha(g)1=\overline{j_\alpha(g,0)}K^\alpha_{g\cdot 0}, \quad g\in G.$$
    Then,
    \begin{align*}
        \ip{{\pi_\alpha}_{1,1}}{{\pi_\alpha}_{1,1}}&=\int_G \ip{1}{{\pi_\alpha}(g)1}_\alpha\ip{{\pi_\alpha}(g)1}{1}_\alpha\haar{g}\\
        &=\int_G |j_\alpha(g,0)|^2\ip{1}{K^\alpha_{g\cdot 0}}_\alpha\ip{K^\alpha_{g\cdot 0}}{1}_\alpha\haar{g} \\
        &= \int_{\B^n} \ip{1}{K_z}_\alpha\ip{K_z}{1}_\alpha \ (1-|z|^2)^\alpha dz\\
        &=\frac{1}{C_\alpha}\int_{\B^n} 1 \ d\mu_\alpha(z) \ \ \text{(by the reproducing formula)}\\
        &=\frac{1}{C_\alpha}.
    \end{align*}
    Thus $d_\pi=C_\alpha$ as required.
\end{proof}

\subsection{Translations  of functions on $\G$ and the ball}\label{ss:Translation}

There are several forms of translation of functions that will play a role in the following. We collect them here for later reference.

We identify right-invariant functions on $G$ with functions on $\B^n$ using the   $G$-equivariant surjection
\[
    p: G  \rightarrow \B^n , \quad
       g  \mapsto g\cdot 0 .
\]
The identification is then given by $F = f\circ p$. We will use this correspondence in the following without further comment. In that connection, uppercase letters and greek letters $\psi$, $\varphi$, etc, will always denote functions on $\G$ and the lower case letter the corresponding function on $\B^n$. The letter $a$ will be reserved for bounded functions on $\G$. All functions will at least be assumed to be measurable. Note that $F$ is smooth if and only if $f$ is smooth. The (usual) translation of a function $F$ on $G$ by $g\in \G$, denoted $\actL_g{F}$, is given by
\[
    (\actL_g F)(h)=(\actL (g))F (h) =F({g^{-1}}{h}), \ \ g,h \in \G.
\]
The right translation of a function $F$ on $\G$  by $g\in \G$,  denoted $\actR_gF$ or $\actR(g)F$, is given by 
\[
    (\actR_g F)(h) = (\actR (g)F)(h)= F(hg), \quad g,h \in \G.
\]
We note that left and right translations commute; that is, $ \actL_{g}\actR_hF =\actR_h \actL_g F$ for all $g,h\in G$. In particular, the left translation maps right $K$-invariant functions into right $K$-invariant functions. Both left and right translations are isometries on $L^p(\G )$, $1\le p \le \infty$, and the left action is an isometry on $L^p(\B^n)$ for all $g\in G$.

The group $G$ also acts on functions on $\B^n$ by $\actL_g f(z) = f(g^{-1}\cdot z)$. We note that $F$ is right $K$ invariant if and only if $\actL_g F$ is right invariant and with the above identification we have $(\actL_gf )\circ p = \actL_gF$.

\subsection{Convolutions of functions}\label{ss:Conv}
Denote by $d\mu_G $  the Haar measure on $\G$ and by $\lambda$   the $G$-invariant measure on $\B^n$. Then, up to a constant,
\begin{equation}\label{eq:InvM}
    d\lambda(z)=\frac{1}{(1-|z|^2)^{n+1}}\dvz.
\end{equation}
With this identification we have
\[
    \int_\G F(g) \haar{g} =\int_{\B^n} f(z) d\lambda(z),\quad  f\in L^1(\B^n,d\lambda).
\]

Consider now the left and right regular representations $\ell$ and $r$ acting on the Banach spaces $L^p(G)$ for each $1\leq p\leq \infty$ that we defined in the previous section.  We can consider the corresponding integrated representations (see Appendix~\ref{sec:integrated_rep}) of $L^1(G)$ on $L^p(G)$. If $\psi\in L^1(G)$ and $F\in L^p(G)$, then one readily sees that $\ell(\psi)F$ corresponds to the usual \textit{convolution} $\psi *F \in L^p(G)$, which we will hitherto refer to as the \textit{left convolution} of $\psi$ and $F$ and denote by $\psi *_\actL F$.

In particular, given two functions $\psi\in L^1(G)$ and $F\in L^p(G)$, the left convolution $\psi*_\ell F$ is given by:
\[
    (\psi *_\actL F)(g):=\actL (\psi )F(g)=\int_{\G} \psi(h)\actL_h F(g) \haar{h} = \int_G \psi(h) F(h^{-1} g) \haar{h}
\]
for all $g\in G$ such that $h\mapsto\psi(h)F(h^{-1}g)$ is integrable (which will occur for almost all $g\in G$).

We similarly define the \textit{right convolution} of $\psi\in L^1(G)$ with $F\in L^p(G)$ as the integrated right regular representation of $\psi$ applied to $F$:
\[
    (\psi *_\actR F)(g):=\actR (\psi )F(g) = \int_{\G} \psi(h)(\actR_h F)(g) \haar{h} = \int_G \psi(h) F(gh) \haar{h}
\]
for all $g\in G$ such that $h\mapsto\psi(h)F(gh)$ is integrable.

More generally, we will use the above formulas to define the left and right convolutions of two measurable functions $\psi,F:G\rightarrow\C$ whenever the integrals converge almost everywhere.

It is well known that the left and right convolutions are connected by $\psi *_\actL F = \psi *_\actR F^\vee$ where $F^\vee (x) = F(x^{-1})$ for all $x\in G$.

Furthermore, the following holds for $x\in \G$ and $\psi$ and $F$ as above 
\begin{align*}
    \actL_x (\psi *_\actL F)
        &=(\actL_x \psi)*_\actL F \quad\text{and} \quad  \actL_x (\psi *_\actR  F)= \psi *_{\actR} (\actL_x F)\\
     \actL_x (\psi*_\actR F) 
        &= \psi *_\actR (\actL_x F )\quad\text{and}\quad  \actR_x (\psi*_\actR F) = \psi *_\actL (\actR_{x}F) .
\end{align*}
The convolution product of functions is associative. However, it is not necessarily commutative, since the group $\G$ is noncommutative.

\subsection{Involutions and lifting to $G$}
\label{sec:involutions_and_lifting}
Recall that $\B^n$ is a Riemannian symmetric space, which means that there are geodesic-inverting isometric involutions around every point in $\B^n$.  For each $w\in\B^n$, let $w_0\in\B^n$ be the midpoint of the (unique) geodesic from $0$ to $w$.  We define $\tau_w : \B^n\rightarrow \B^n$ to be the isometric involution that fixes the point $w_0$ and inverts geodesics around $w_0$.  In particular, one sees that $\tau_w 0 = w$ and that $\tau_w w = 0$ (see, for instance, \cite[Section 1.2]{Z05}).

In the case of the unit disk (i.e., $n=1$), we can write:
\[
    \tau_z(w) := \frac{w-z}{\langle w, z \rangle -1}.
\]
In fact, one quickly checks that $\tau_z$ is an involution, that $\tau_z(0) = z$ and $\tau_z(z) = 0$ for all $z\in\B^n$.  One can also verify that $\tau_z$ fixes the point $w_0 = \frac{1-\sqrt{1-|w|^2}}{|w|^2}z$ and inverts geodesics around $w_0$.

Returning to the the general case, it is not difficult to see that these involutions $\tau_z$ live in the connected component of the identity of the isometry group $\mathrm{Iso}_0(\B^n)$ of $\B^n$, which can be identified with $G/Z$, where $Z=\{\omega I \mid \omega^{n+1}=1\}$ is the center of $G=\SU(n,1)$.  The map
\begin{align*}
    \tau : \B^n & \rightarrow G/Z \\
            z   & \mapsto     \tau_z
\end{align*}
can be lifted to a smooth map $\tau' : \B^n \rightarrow G$ satisfying that $\tau'_z \cdot w = \tau_z (w)$ for all $z,w\in\B^n$.  Furthermore, once we fix the value of $\tau'_0$, the lifting is unique, as usual.  If $\tau', \tau'': \B^n \rightarrow G$ are both liftings of $\tau$, then there exists an element $\omega$ of $Z$ such that $\tau''(z) = \omega \tau'(z)$ for all $z\in\B^n$.

\subsection{Convolutions of radial functions}

Recall that if $K$ is a compact subgroup of a locally-compact group $G$, then we say that $G/K$ is a \textit{commutative space} if the Banach algebra $L^1(K\backslash G / K)$ of bi-$K$-invariant functions is commutative under the convolution product.  Note that for $1\le p \le \infty$ one can identify 
\[ 
    L^p(K \backslash G /K) = \{ f\in L^p(G) \mid f(k_1 g k_2) = f(g) \text{ for almost all } k_1, k_2 \in K \text{ and } g\in G \}
\]
with the space
\[ 
    L^p( G / K)^K = \{ f\in L^p(G/K) \mid f(k\cdot x) = f(x) \text{ for almost all } k\in K \text{ and } g\in G \}
\]
of \textit{radial} (left-$K$-invariant) functions on $G/K$. It is well known that the unit ball $\B^n = \SU(n,1) / \U(n)$ is a commutative space.  Recall that a function $f$ on $\B^n$ is said to be radial if and only if there exists a function $\varphi$ defined on $[0,\infty)$ such that $f (z) = \varphi (|z|)$ for all $z\in \B^n$. In particular, $f(z)$ only depends on the distance from $z$ to the origin.

We define a projection $L^p (\B^n)\to L^p(\B^n)^K, f\mapsto f^\#$ by
\[
    f^\sharp (x) =\int_K f(kx)\, d\mu_K(x),
\]
where $\mu_K$ is the normalized Haar measure on $K$ (that is, $\mu_K (K)=1$). We also define the projection $L^p(\G) \to L^p(\B^n), f \mapsto f_K$ by
\[
    f_K (x) = \int_K f(xk)d\mu_K(k)
\]
for all $x\in \B^n$.
 
We recall the following property of radial functions:
\begin{lemma}\label{lem:radialf_invar}
    Let $f$ be a radial function on $\B^n$. Then 
    \[
        F(g)=F(g^{-1}),\quad g\in \G,
    \]
    where $F = f\circ p$ as usual.  Also,
    \[
        \actfL{g}{f}=\actftL{z}{f},
    \]
    where $z=g\cdot 0$.
\end{lemma}

\begin{proof}
    Let $g\in \G$ and let $z_g=g\cdot 0$. Then, since 
    $\Act{z_g}(z_g) = \Act{z_g}(g\cdot 0) = 0$, we have that $\Act{z_g}g=k_g$ for some $k_g\in K$, where we view $\Act{z_g}$ and $g$ as elements of $\mathrm{Aut}(\B_n)$. It follows that $g=\Act{z_g}k_g$ and thus that
    \[
        f(g^{-1}\cdot w)=f(k_g^{-1}\act{z_g}{w})=a(\act{z_g}{w})
    \]
    for all $w\in \B^n$, since $f$ is radial. Finally, by taking $w=0$, we get that
    \[
        f(g^{-1})=f(g^{-1}\cdot 0)=f(\act{z_g}{0})=f(g\cdot 0)=F(g).\qedhere
    \]
\end{proof}

\begin{lemma}\label{lem:2.4} Let $1\le p \le \infty$ and let $p^\prime = p/(p-1)$. Let $f\in L^p (\B^n)$ and $g\in L^{p'}(\B^n)$ or
$f,g\in L^1(\B^n)$. Then the following holds:  
\begin{enumerate}
    \item If $g$ is $K$-invariant and $h\in L^p(\G)$ or in $L^1(\G )\cup L^\infty (\G)$ if $p=1$, then $h *g =h_K * g$.
    \item $f*g = f*g^\sharp $ and 
    \[ f*g(w) = \int_{\B^n} f(z) g^\sharp (\tau_z w) \inv{z}.\]
    \item If $f$ is $K$-invariant, then $f*g$ is $K$-invariant.
\end{enumerate}
\end{lemma}

\begin{proof} 
    This is well known, but we include the proof for completeness. 
    
    \noindent
    (1) We have, for all $k\in K$, that
    \begin{align*}
    h*g(w) &=\int_\G h(x )g (x^{-1}w)\haar{x} = \int_\G h(x )g ( kx^{-1}w)\haar{x}\\
           &=\int_\G h(xk )g(x^{-1}w)\haar{x} .
    \end{align*}
    Now integrate over $K$.
    \medskip

    \noindent
    (2)
      Let $k\in K$ and $w\in \B^n$. Then
    \begin{align*}
        f*g (w) &= \int_{\G} f(x) g(x^{-1}w)d\mu (x) =\int_{\G} f(xk) g(x^{-1}w)d\mu (x)\\
                &=\int_{\G} f(x) g(kx^{-1}w)d\mu (x)
    \end{align*}
    Now integrate over $K$. 

    \noindent
    (3) 
    This claim follows from \eqref{eq:InvM}.
    \[
        f*g(kw) =\int_G f(x) g(x^{-1} kw)\haar{x} =\int_G f(kx) g(x^{-1} w)\haar{x} =f*g(w).\qedhere
    \]
\end{proof}

\subsection{The natural approximate identity arising from $\Bergalph$}
\label{sec:natural_approximate_identity}
For all $\alpha > -1$, let $\varphi_\alpha$ be given by 
\[
    \varphi_\alpha(z)=C_\alpha(1-|z|^2)^{\alpha+n+1} 
    = \frac{C_\alpha}{K^\alpha (z,z)} 
    = C_\alpha |j_\alpha (\tau_z,0)|^2
    = C_\alpha |j_\alpha (g,0)|^2,
\]
where $z= g\cdot 0$ and where $C_\alpha=\frac{(n+\alpha)!}{n!\alpha!}$ is chosen so that $\int_{\B^n} \varphi_\alpha (z)\ \inv{z}=1$. In particular $C_0 =1$
and
\begin{equation}\label{eq:vphi0}
    \varphi_0(z) =\varphi(z)= (1-|z|^2)^{n+1}=|j(g,0)|^2 = \|K_z\|^2,
\end{equation} 
where $z=g\cdot 0$. 
\begin{lemma}\label{lem:2.5}
    We have the following identities:
    \begin{enumerate}
        \item  
        $\displaystyle \varphi_\alpha(g\cdot z)=|j_\alpha(g,z)|^2\varphi_\alpha(z)$
        \item
        $\displaystyle \varphi_\alpha(g^{-1}\cdot z)=|K_\alpha(z,g\cdot 0)j_\alpha(g,0)|^2\varphi_\alpha(z)=|\keralph{g\cdot 0}{z}|^2\varphi_\alpha(z)$.
    \end{enumerate}
\end{lemma}

\begin{proof} This follows from Lemma~\ref{lem:cocycle_normalized_ker}.
\end{proof}

In fact, one can show that the functions $\{\varphi_\alpha\}_{\alpha>0}$ form 
a right approximate identity for the convolution algebra $L^1(G/K) =\Lone$, in the sense that $f*\varphi_\alpha \rightarrow f$ in the $L^1$ norm as $\alpha\rightarrow \infty$ for all $f\in L^1(G/K)$.  On the other hand, the corresponding functions $\{\varphi_\alpha\}_{\alpha>0}$ do \textit{not} form an approximate identity for $L^1(G)$. In fact, Lemma~\ref{lem:2.4} implies that for $f\in L^1(G)$, then one has that $\varphi_\alpha *f  =  \varphi_\alpha * f^\sharp$, which then converges to $f^\sharp$ as $\alpha \rightarrow \infty$. However, $\{\varphi_\alpha\}_{\alpha>0}$ does form a left approximate identity in $\Lone^K$ as this algebra is commutative.

We recall the Cohen-Hewitt factorization theorem (see \cite[Thm.\ 1]{C59} and \cite[Thm.\ 2.5]{H64}):
\begin{theorem}[Cohen-Hewitt]
    Let $A$ be a Banach algebra with a left approximate identity. Let $T$ be a continuous representation of $A$ on a Banach space $X$. Then $T(A)X$ is a closed subspace of $X$. 
\end{theorem}
This, together with Lemma~\ref{lem:2.4}, now implies that
\begin{lemma}\label{lem:2.9} 
    Let $1\leq p<\infty$. Then $L^1(\B^n)*L^p(\B^n)^K= L^p(\B^n)^K$. 
\end{lemma}

The $\alpha$-Berezin transform of a function $a :\B^n\to \C$ such that $a|k^\alpha_z|^2$ is integrable for all $z\in \B^n$ is just the Berezin tranform of $T_a$ in $\cA^2_\alpha(\B^n)$; that is, it is given by
\begin{equation}
    \label{eq:BerAlphaDefinition}
    B_\alpha(a) (z)=\ip{a k^\alpha_z}{k^\alpha_z}_\alpha 
            =C_\alpha (1-|z|^2)^{n+1+\alpha}\int_{\B^n}a(z)\frac{(1-|w|^2)^\alpha}{|1-z\overline w|^{2(n+1+\alpha)}} \dv(w)
\end{equation}
for all $z\in \B^n$, where $\ip{\cdot}{\cdot}_\alpha$ is the inner-product in $\Bergalph$. 

\begin{lemma}\label{lem:Berezin of a function}
    Let $\alpha>0$ and let $a :\B^n\to \C$ be a function such that $B_\alpha(a)$ is defined. Then
 \[B_\alpha(a)=a\last \varphi_\alpha.\]  
\end{lemma}

\begin{proof}
This follows from (1) in Lemma~\ref{lem:2.5}.
\end{proof}


\section{QHA on the Bergman space}\label{sec:QHA Bergman}

In this section, we explore QHA on Bergman spaces and
specialize to the case $\cH = \cA^2(\B^n)$.
A comprehensive discussion of QHA on locally compact abelian groups is provided in Appendix \ref{sec:general QHA}.
 We write Toeplitz operators and the Berezin transform as convolutions with a finite rank operator $\Phi$.

 \subsection{Left translations of operators}
 \label{sec:left_translations_of_operators}
The irreducible square-integrable unitary representation $\piL$ of $\G$ acting on $\Berg$ facilitates QHA on 
the Bergman space and inherits the theory discussed in Appendix \ref{sec:general QHA}. 

To begin, we define the map $\pi:\B^n \to \bdd$ by 
\begin{align}\label{def:piz}
    \pi(z)f(w)  & = |j(\tau_z^{-1}, w)|f(\tau_z^{-1}\cdot w) = |j(\tau_z,w)|f(\tau_z\cdot w) \nonumber \\
                & = |j(\tau_z,0)|K_z(w)f(\tau_z\cdot w)
\end{align}
for all $f\in\Berg$ and $w\in\B^n$ (recall that $\tau_z$ is an involution). Here we have used Lemma~\ref{lem:cocycle_normalized_ker} in the last equality. Compare this with the operator $\pi(\tau_z)$, which is defined by:
\begin{align}
    \pi(\tau_z)f(w) & = j(\tau_z^{-1}, w) f(\tau_z^{-1} \cdot w) = j(\tau_z, w) f(\tau_z \cdot w) \nonumber \\
                    & = \overline{j(\tau_z,0)}K_z(w)f(\tau_z\cdot w)
\end{align}
for all $f\in \Berg$ and $w\in \B^n$,  We note that 
\[
    \pi(z) = \frac{j(\tau_z,0)}{|j(\tau_z,0)|} \pi(\tau_z)
\]
for all $z\in\B^n$.  It follows immediately that $\pi(z)$ is a unitary operator.  

Furthermore, we note that $\pi(z)$ is an involution.  In fact, $|j|:G\times \B^n \rightarrow \R_{>0}$ is a cocycle (as follows immediately from the fact that $j$ is a cocycle), and thus 
\[
    \pi(z)\pi(z)f(w) = |j(\tau_z,w) j(\tau_z,\tau_z \cdot w)| f(\tau_z\tau_z\cdot w) = f(w) 
\]
for all $f\in\Berg$ and $z,w\in\B^n$.

We also note that if $z = g\cdot 0$ with $g\in G$, then $\pi (g)$ and $\pi (z)$ differ by a phase factor because of the absolute value in \eqref{def:piz}. Finally, the absolute value in \eqref{def:piz} shows that $\pi (z)$ does not depend on the lifting of $\tau_z$ from $\mathrm{Aut} (\B^n)$ to $\G$ that we chose.

This allows us to define the translation of an operator $S$ by $z\in \B^n$ to be the operator 
\[
    \tpi(z)S:= \pi(z)S\pi(z)^{-1} = \pi(z)S\pi(z). 
\]
Despite the fact that $\pi(z)$ and $\pi(\tau_z)$ are not equal in general, it \textit{is} true that $\tpi (z) = \tpi (\tau_z)$ for all $z\in \B^n$ because the phase factors cancel out.

\subsection{Convolutions of operators.}
\label{sec:convolutions_of_operators}
For a function $\psi$ on $\G$ and $S\in \bdd$, the convolution $\psi\Gast S$ is defined formally by the weak-integral:  
\[
   \psi\Gast S := \tpi (\psi) S =\int_{\G} \psi(g) \tpi (g)S \haar{g}.
\]
For $1\leq p<\infty$, the space of all Schatten class operators on $\Berg$, is denoted by $\schop$. For the case $p=\infty$, we make the identification $\cS^\infty(\cA^2)=\bdd$.
Then for $1\leq p\leq \infty$, the following holds:
\begin{enumerate}
    \item If $\psi\in \LoneG$ and $S\in \schopq{p'}$the convolution $\psi\Gast S$ is well-defined and satisfies the following QHA Young's inequality:
    \[
        \|\psi\Gast S\|_p\leq \|\psi\|_1\|S\|_p.
    \]
    \item If $\psi\in \Lp$ and if $S\in \traceop$, the convolution $\psi\Gast S$ is well-defined and satisfies the following QHA Young's inequality:
    \[
        \|\psi\Gast S\|_p\leq \|\psi\|_p\|S\|_1.
    \] 
    \item Let $\psi_1,\psi_2$ are functions on $\G$ and $S\in \bdd$ s.t. the convolutions below are well defined by the QHA Young's inequality. Then
    \[
        (\psi_1\Gast\psi_2)\Gast S=\psi_1\Gast(\psi_2\Gast S).
    \]
\end{enumerate}
\begin{remark}
    In the above inequalities, we have used the fact that the formal dimension of $\pi$ is given by $d_\pi=1$. 
\end{remark}
For $S\in \cS^p$ and $T\in \cS^{p'}$, we define the \textit{(twisted) convolution} $S\ast A$ to be the function on $\G$ given by
\[
    (S\Gast A)(x)=\Tr(S\tpi (x){A}),\quad x\in \G,
\]
Let $1\leq p\leq \infty$. If $S\in\schop$ and $A\in \traceop$, $S\ast A\in \schop$ and we have
\[
    \|S\ast A\|_p\leq \|S\|_p\|A\|_1.
\]
For more details and proofs of these relations, see Appendix~\ref{sec:general QHA}.

\subsection{Radial operators}
Recall that an operator $S\in \bdd$ is \textit{radial} or $K$-invariant if 
\[
    \actopLnoouterparens{k}{S}=\pi (k)S\pi (k)^{-1}=S, \quad \text{ for all } k\in K.
\]
We denote the set of all radial operators, which forms a closed subalgebra of $\bdd$, by $\rad$, and we denote the closed subspace of all radial Schatten-class operators by $\schoprad$. It is known that an operator $S\in \bdd$ is radial if and only if its Berezin transform $B(S)$ is radial.   A Toeplitz operator $T_a$ is radial if and only if the symbol $a$ is a radial function. The \textit{radial Toeplitz algebra} $\toepalgrad$ is the closed subalgebra of $\toepalg$ generated by radial Toeplitz operators. It is also known that $\toepalgrad=\toepalg\cap \rad$. 

Recall that the radialization of an operator $S\in \bdd$ is given by
\[
    \Rad{S}=\int_{K} \piL(k)S\piL(k)^{-1} \haarK{k} 
\]
where $d\mu_K$ denotes the normalized Haar measure on $K$. Then $\|\Rad{S}\|\leq \|S\|$.

\begin{lemma}\label{lem:convolution_rad}
    Let $\psi$ be a function on $\B^n$ and let $S\in \bdd$ such that the convolution $\psi\Last S$ is defined. Then $\psi \Last \Rad{S}$ is defined and $$\psi \Last \Rad{S}=\psi \Last S.$$
\end{lemma}

\begin{proof}
    We identify functions in $\Lone$ with the corresponding right-$K$-invariant functions on $G$. Suppose that $\psi\in\Lone$ and $S\in \bdd$.
    One easily checks that
    \begin{align*}
        \int_{\G}  \int_{K} |\psi(g)| |\ip{\piL(g)\piL(k)S\piL(k)^{-1}\piL(g)^{-1}f_1}{f_2}| \ \haarK{k} \haar{g} 
        \leq \|\psi\|_1\|S\|\|f_1\|\|f_2\| < \infty.
    \end{align*}
    We can thus apply Fubini's theorem to the above double integral without absolute values, and it follows that     
    \begin{align*}
        \psi \Last \Rad{S} &= \int_{\G} \psi(g) \actopL{g}{\Rad{S}} \haar{g}\\
        &= \int_{\G} \psi(g) \int_{K} \piL(g)\piL(k)S\piL(k)^{-1}\piL(g)^{-1} \haarK{k} \haar{g}\\
        &= \int_{K} \int_{\G} \psi(g)  \piL(gk)S\piL(gk)^{-1}  \haar{g} \haarK{k}\ \ \ \text{(by Fubini)}\\
        &= \int_{\G}  \int_{K} \psi(gk^{-1}) \haarK{k} \actopL{g}S \haar{g}\\
        &= \int_{\G} \psi(g) \actopL{g}S \haar{g}\\
        &= \psi \Last S. \qedhere
    \end{align*}
 \end{proof}

As a consequence of the above lemma, we have the following corollary.
\begin{corollary}
    If $\psi\in \Lone$ is radial, then $\psi\Last S$ is radial.
\end{corollary}

\begin{lemma}\label{lem:ConAst} 
    Suppose that $S$ is radial and that $\psi \in L^p(\G)$ for some $1\le p \le \infty$ such that $\psi \ast_\pi S$ exists. Then $\psi \ast_\pi S = \psi^\# \ast_\pi S$.
\end{lemma}
\begin{proof} 
    This follows in the same way as above by noting that for all $k\in K$,
    \begin{align*} 
        \psi\ast_\pi S & = \int_\G \psi (x)\wpi (x)S\haar{x} 
          = \int_\G \psi (xk)\wpi (x)S\haar{x} \\
        & = \int_\G \psi (x)\wpi (xk^{-1})S\haar{x}
          = \int_\G \psi (x)\wpi (x)S\haar{x}
    \end{align*}
and then averaging over $K$.
\end{proof}

\begin{lemma}\label{lem:radialop_invar}
    Let $S\in \bdd$ be radial. Then for all $g\in \G$,
    \[
        \actopLnoouterparens{g}{S}=\actopLnoouterparens{g\cdot 0}{S}.
    \]
\end{lemma}

\begin{proof}
    Note that for all $g\in \G$, we have $g=\Act{g\cdot 0}k_g$ for some $k_g\in K$.  Now write $z=g\cdot 0$. Then, since $S$ is radial, 
    \[
        \actopLnoouterparens{g}{S}=\piL(g)S\piL(g^{-1})=\piL(\Act{z})\piL(k_g)S\piL(k_g^{-1})\piL(\Act{z})=\piL(\Act{z})S\piL(\Act{z}). \qedhere
    \]
\end{proof}

\subsection{Toeplitz operators as convolutions}

Now consider the rank-one operator $\Phi=\one\otimes \overline{\one}\in\bdd$. In the following we write $\one$ for $\one_G$.
\begin{lemma}\label{lem:PhiK0} 
    Let $f\in \cA (\B^n)$. Then $\Phi f = f(0)\one$. 
\end{lemma}
\begin{proof} 
    Write $f = f(0)\one + \sum_{m\not=0} \ip{f}{e_m}{e_m}$ where $\{e_m\}$ is the orthogonal basis \eqref{eq:onb}. Then $\Phi (f) = f(0)\one + \sum_{m} \ip{f}{e_m}\ip{e_m}{\one}\one= f(0)\one $.
\end{proof}

Or goal is to show that $T_a = a*\Phi$ (see Lemma~\ref{lem:Toeplitz_convolution}).  

\begin{lemma}\label{lem:R_translation}
    Let $g\in \G$ and $f\in \Berg$. Then
    $
        \actopL{g}{\Phi}f=\varphi(g\cdot 0)f(g\cdot 0)K_{g\cdot 0}.
    $
    In particular, $\actopL{k}{\Phi} f =\Phi f$ for all $k\in K$, and $\Phi$ is radial.
\end{lemma}

\begin{proof}
    Let $f\in \Berg$ and $g\in \G$. By Lemma~\ref{lem:PhiK0} we have   $\Phi f=f(0)\one$. Hence 
    \[
        \Phi \piL(g^{-1})f=\ip{\piL(g^{-1})f}{\one}=j(g,0)f(g\cdot0).
    \]
    Thus
    \begin{align*}
        \actopL{g}{\Phi}f(z) &= (\piL(g)\Phi\piL(g)^{-1})f(z)\\
        &= j(g^{-1},z)j(g,0)f(g\cdot 0)\\
        &=K_{g\cdot 0}(z)|j(g,0)|^2 f(g\cdot 0) \ \ \text{ (by Lemma~\ref{lem:cocycle_normalized_ker})}\\
        &=K_{g\cdot 0}(z)\varphi(g\cdot 0) f(g\cdot 0). \qedhere
    \end{align*}
    If $k\in K$ then $\varphi (k\cdot 0) f(k\cdot 0) K_{k\cdot 0} = f(0)\one = \Phi (f)$.
\end{proof}

\begin{lemma}\label{lem:Toeplitz_convolution}
    Let $a\in \bddf$. Then 
    $
         T_a = a\ast_\pi  \Phi 
    $.
\end{lemma}

\begin{proof} 
    Let $f_1,f_2\in \Berg$ and $g\in G$. Then, by Lemma~\ref{lem:R_translation}:  
    \[
        \ip{\actopL{g}{\Phi}f_1}{f_2}=\varphi(g\cdot 0) f_1(g\cdot 0)\ip{K_{g\cdot 0}}{f_2}\\
            = \varphi(g\cdot 0) f_1(g\cdot 0)\overline{f_2(g\cdot 0)}.
    \]
    It follows that
    \begin{align*}
        \ip{(a\Last  \Phi)f_1}{f_2} &= \int_\G a(g\cdot 0) \ip{\actopL{g}{\Phi}f_1}{f_2} \haar{g}\\
        &= \int_{\B^n} a(w)\varphi(w)f_1(w)\overline{f_2(w)} \inv{w}\\
        &= \int_{\B^n} a(w)f_1(w)\overline{f_2(w)} \ dA(w)\\
        &= \ip{T_af_1}{f_2},
    \end{align*}
which implies the claim.  
\end{proof}

\subsection{Berezin transform of an operator as a convolution}
In this section, we show that the Berezin transform of an operator $S\in \bdd$ can be written as a convolution of operators.  First, we need to calculate the result of taking the convolution of $\Phi$ with $\Phi$:

\begin{lemma}
    Recall the finite-rank operator $\Phi=\one \otimes \overline{\one}$. Then $\Phi\Last \Phi=\varphi$.
\end{lemma}

\begin{proof} 
    Let $\{e_{\textbf{m}}\}_{\bm}$ be the orthonormal basis as in \eqref{eq:onb}. Using that $\Phi^* = \Phi$ we get, for all $g\in\G$, that
    \begin{align*}
        \Tr(\Phi\actopL{g}{\Phi}) & = \sum_{\bm} \ip{\Phi \actopL{g}{\Phi} e_\bm}{e_\bm}\\
                                  & =\sum_{\bm} \ip{ (\tens{\pi (g)\one}{\pi (g)\one}) e_\bm}{\Phi e_\bm}\\
                                  & =\ip{(\tens{\pi (g)\one}{\pi(g)\one}) \one}{\one}\\
                                  & =|j(g^{-1},0)|^2,
    \end{align*}
    and the claim follows from \eqref{eq:vphi0} and Lemma~\ref{lem:radialf_invar}. 
\end{proof}

We now realize the Berezin transform as convolution with the rank-one operator $\Phi$. 
\begin{lemma}
    Let $S\in \bdd$. Then $
        B(S)=S\Last \Phi$.
\end{lemma}

\begin{proof} 
    Let $\{e_{\textbf{m}}\}_{\bm}$ be the orthonormal basis as in \eqref{eq:onb}. Then
    \begin{equation}\label{eq:Kz}
        K_z = \sum_{\bm} \overline{e_\bm(z)}e_\bm 
    \end{equation}
    by \eqref{eq:Konb}. It follows, with $z=g\cdot 0$, that
    \begin{align*}
        \Tr (S\actopL{g}{\Phi}) &= \sum_{\bm} \ip{S\actopL{g}\Phi e_\bm}{e_\bm}\\
        &= \sum_{\bm} \ip{\actopL{g}\Phi e_\bm}{S^*e_\bm}\\
        &= \sum_{\bm} \varphi (z) e_\bm (z) \ip{K_z}{S^*e_\bm}\\
        &=   \varphi (z) \ip{SK_z}{K_z} \qquad\text{by \eqref{eq:Kz}}\\ 
        &=B(S)(z) \qquad \text{by \eqref{eq:vphi0} and \eqref{eq:BerS}} \qedhere.
    \end{align*}
\end{proof}

\subsection{Properties of radial convolutions} 
Here we observe certain properties of convolutions of radial operators. Some of these properties stem from the fact that $\G/K$ is a commutative space. To prove the following lemma, we use convergence properties of QHA convolutions, which are discussed in detail in Appendix \ref{sec:general QHA} (in particular, see Lemma \ref{lem:convergence}).

\begin{lemma}\label{lem:associativity_radial}
    Let $S\in \bdd$ and $A,B\in \traceop$. Then  
    \begin{enumerate}
        \item If $H\subseteq \G$ is a subgroup and $S$ is $H$ invariant, then $S\Last A$ is left $H$ invariant and $A\Last S$ is right $H$ invariant. In particular, if $S$ and $A$ are  radial, then $S\Last A$ is a radial function on $\B^n$.     
        \item If $S$ and $A$ are radial, then $S\Last A = A\Last S$.
        \item If $A$ and $B$ are radial, then $(S\Last A)\Last B= (S\Last B)\Last A$.   
        \item If $A$, $B$, and $S$ are radial, then $(S\Last A)\Last B= (A\Last B)\Last S$.   
    \end{enumerate}
\end{lemma}
\begin{proof}
    To prove (1), note that for all $h\in H$ and $g\in G$, 
    \begin{align*}
        S\Last A(h^{-1} g) &= \Tr(S\piL(h)^{-1}\piL(g)A\piL(g)^{-1}\piL (h) )\\
        &=\Tr(\piL(h)S\piL(h)^{-1}\piL(g)A\piL(g)^{-1})=(S\Last A)(g)
    \end{align*}
    and
    \[
        A\Last S (gh) = \Tr( A \pi (g) \piL(h)S \piL(h)^{-1}\piL(g)^{-1} )   
                      = \Tr( A \pi (g) S \piL(g)^{-1} )
                      = A\Last S (g).
    \]
    \medskip

    \noindent
    (2) We first use the fact, by \eqref{eq:operator_convolution_commutation}, it follows that $S\Last A(g) = A\Last S(g^{-1})$ for all $g\in G$.  Furthermore, we know that $S\Last A$ is radial by (1).  Lemma~\ref{lem:radialf_invar} then shows that $S\Last A = A\Last S$. 
    \medskip

    \noindent
    (3) First we consider $A=T_a$ and $B=T_b$, where $a$ and $b$ are radial functions in $\bddf$. Then
    \[
        (T_a\Last \Phi)\Last T_b=(a\Last \varphi)\Last (b\Last \Phi)=(a\Last \varphi\Last b )\Last \Phi=(b\Last \varphi\Last a )\Last \Phi
    \]
    by the commutativity of the convolutions of radial functions. Hence
    \[
        (T_a\Last \Phi)\Last T_b=(T_b\Last \Phi)\Last T_a
    \]
    Then, by the SOT density of radial Toeplitz operators in all radial operators (see Lemma~\ref{lem:convergence}), we have the result for general $B$:
    \[
        (T_a\Last \Phi)\Last B=(B\Last \Phi)\Last T_a,
    \]
    and a similar argument shows that
    \[
        (A\Last \Phi)\Last B=(B\Last \Phi)\Last A.
    \]
    Now, by taking $S=T_s$, $s\in \bddf$, we have from the above equality that
    \begin{align*}
        (T_s\Last A)\Last B &=(s\Last (\Phi\Last A))\Last B=s\Last ((A\Last \Phi)\Last B)\\
        &=s\Last ((B\Last \Phi)\Last A)=(T_s\Last B)\Last A.
    \end{align*}
    Therefore, we have $(S\Last A)\Last B= (S\Last B)\Last A$ by applying once more the SOT density of Toeplitz operators $\bdd$.   
    \medskip

    \noindent
    (4) By (2) and (3), we get that $(S\Last A)\Last B= (A\Last S)\Last B=(A\Last B)\Last S$, from which the claim follows.
\end{proof}

\section{\texorpdfstring{$\G$}{G}-uniform continuity of functions and operators}

\subsection{Left $G$-uniform continuity}
Here, we discuss left $G$-uniform continuity of functions and operators.
For $S\in \bdd$, the map $g\mapsto \actopLnoouterparens{g}{S}$ is continuous in the strong operator topology, but not necessarily in the norm topology.

\begin{definition}
    \label{def:cbul:functions}
    A function $a\in \bddf$ is said to be \textit{left-$\G$-uniformly continuous} if the map $G\rightarrow \bddf$, $g\mapsto \actfL{g}{a}$, is continuous  with respect to $\|\cdot\|_\infty$. We denote the set of all bounded  left-$\G$-uniformly continuous functions by $\cbufL$. 
\end{definition}

\begin{definition}
    \label{def:cbul:operators}
    An operator $S\in \bdd$ is said to be \textit{left-$\G$-uniformly continuous} if the map $G\rightarrow \bdd$, $g\mapsto \actopLnoouterparens{g}{S}$, is continuous with respect to the operator norm. The set of all bounded left-$\G$-uniformly continuous operators is denoted by $\cbuopL$.
\end{definition} 

The spaces $\cbufL$ and $\cbuopL$ are $C^*$-subalgebras of $\bddf$ and $\bdd$, respectively.

If  $\psi\in \LpG$, where $1\leq p<\infty$, and $a$ is a function on $\B^n$ such that $\psi\Last a$ is defined, then we have $\psi\Last a\in \cbufL$. This is because
\[
    \actfL{g}{(\psi\Last a)}=(\actfL{g}{\psi})\Last a
\]
and because $\G$ acts continuously on $\LpG$.

\begin{lemma} The following statements are true:
   \begin{enumerate}
        \item All compact operators are in $\cbuopL$.
        \item If $\psi\in \LoneG$ and $S\in \bdd$, then $\psi\Last S\in \cbuopL$.
        \item If $\psi\in \cbufL$ and $S\in \traceop$, then $\psi\Last S\in \cbuopL$.
    \end{enumerate}
\end{lemma}

\begin{proof}
    To prove (1), it is enough to show that all finite-rank operators are in $\cbuopL$. This is true by the strong continuity of the representation $\pi$; in fact, for all $f_1,f_2\in \Berg$ and $g_1,g_2\in \G$,
    \begin{align*}
        \|\actopL{g_1}{f_1\otimes f_2}-\actopL{g_2}{f_1\otimes f_2}\|
        \leq &\,\| (\pi(g_1)f_1-\pi(g_2)f_1)\otimes \pi(g_1)f_2 \| \\
             & + \|\pi(g_2)f_1\otimes (\pi(g_1)f_2-\pi(g_2)f_2)\|\\
        \leq &\, \|f_2\|\|\pi(g_1)f_1-\pi(g_2)f_1\|+ \|f_1\|\|\pi(g_1)f_2-\pi(g_2)f_2\|.
    \end{align*}
    To prove (2) and (3), we note that $\actopLnoouterparens{g}{(\psi\Last S)}=(\actfL{g}{\psi}) \Last S$ and that $G$ acts continuously on both spaces $\LoneG$ and $\cbufL$.
\end{proof}

We state the theorem below for the sake of completeness, and the proof, which we have omitted, follows the same ideas as in the proof of Proposition 2.16 in \cite{F19}. Given a subset $W$ of a topological space $X$, we denote by $\overline{W}^X$ the closure of $W$ in $X$.
\begin{proposition}\label{prop:cbu_characterizations}
    Let $\psi\in \LoneG$ be a cyclic vector for $\LoneG$ and let $\{\psi_\alpha\}_{\alpha\in\N}$ be an approximate identity for $\LoneG$. Then we have the following characterizations of $\cbuopL$:
    \begin{align*}
        \cbuopL&= \overline{\LoneG \Last \bdd}^{\cB} \\
        &= \{S\in \bdd \mid \psi_\alpha \Last S \to S \text{ in norm}\}\\
        &=\overline{\psi\Last \bdd}^{\cB}\\
        &=\LoneG \Last \cbuopL.
    \end{align*}
\end{proposition}

\subsection{Right translations and convolutions of functions}

We define the right translation of a function $f$ on $\B^n$ by $g\in \G$, denoted $\actfR{g}{f}$, 
is  the function on $\G$ given by
\[
    (\actfR{g}{f})(h)=f(hg\cdot 0), \ \ \forall h\in \G.
\]
Unlike the left translation, the right translation $\actfR{g}{f}$ is generally not a right-$K$-invariant function on $G$ and hence
does not define a function on $\B^n$. 
We define the right translation of $f$ by $z\in \B^n$ as follows:
\[
    \actftR{z}{f}:=\actfR{\tau_{z}}{f},
\]
where $\tau_z$ is as before the involution interchanging $0$ and $z$.
Since each $g\in \G$ can be written as $g=\tau_{g\cdot 0}k_g$ for some $k_g\in K$, we have that  
\[
    \actftR{g\cdot 0}{f}=\actfR{g}{f}.
\] 

We define right convolution analogously to left convolution.  One shows that if $\psi$ and $a$ are functions on $\B^n$ (i.e., right-$K$-invariant functions on $G$) such that the convolution $\psi \Rast a$ is defined, 
then $\psi \Rast a = a \last (\psi^\vee)$ is the function on $\G$ given by
\begin{align*}
    (\psi\Rast a)(g)&=\int_{\G} \psi(h\cdot 0)a(gh\cdot 0) \haar{h}\\
    &= \int_{\B^n} \psi(w)a(gw) \inv{w}.
\end{align*}
If, in addition, $\psi$ is a radial function (i.e., a bi-$K$-invariant function on $G$), then $\psi\Rast a$ is a function on the ball $\B^n$ (i.e., a right-$K$-invariant function on $G$), and $\psi \Rast a =a \last \psi$.

\subsection{Right-$G$-uniformly continuous functions}
We now introduce right-$G$-uniformly continuous functions.

\begin{definition}
    \label{def:cbur:functions}
    A function $a$ on $\B^n$ is \textit{right-$\G$-uniformly continuous} if the map $G\rightarrow \bddfG$, $g\mapsto \actfR{g}{a}$, is continuous  with respect to $\|\cdot\|_\infty$. We denote the set of all bounded, right-$\G$-uniformly continuous functions by $\cbufR$. 
\end{definition}
It is not difficult to see that $\cbufR$ is a $C^*$-subalgebra of $\bddf$.
\begin{lemma}\label{lem:function_convol} 
     Let $\psi$ and $\varphi$ be functions on $G$. Then 
    \begin{enumerate}
    \item If $\psi\in \LoneG$ and $\varphi\in L^\infty(\B^n)$ then $\psi\last \varphi\in \cbufL$ and 
    $\|\psi *_\ell \varphi\|_\infty \leq \|\psi\|_1\|\varphi\|_\infty.$
    \item If $\varphi \in L^\infty(\B^n)$ and $\psi\in \Lone$ then $ \varphi \last\psi\in \cbufR$ and
    $\|\varphi \last\psi\|_\infty \leq \|\varphi\|_\infty \|\psi\|_1.$
    \item If $\psi\in L^p(G)$ and $\varphi\in L^{p'}(\B^n)$ where $1<p<\infty$ and  $\frac{1}{p}+\frac{1}{q}=1$, then $\psi \last\varphi\in \cbufL\cap \cbufR$ and
    $\|\psi \last \varphi\|_\infty \leq \|\psi\|_p \|\varphi\|_q.$
\end{enumerate}
\end{lemma}

\begin{proof}
    The proof follows from the behaviour of translations of convolutions and by the fact that $\G$ acts continuously on $L^p(G)$ with respect to the SOT for $1\leq p <\infty$. Also, the norm inequalities given are well known and easy to verify.
\end{proof}

Recall the approximate identity  $\{\varphi_\alpha\}_{\alpha\in \N_0}$ of $\Lone$. Since each $\varphi_\alpha$ is a radial function, we have that $\varphi_\alpha \Rast f=f\last \varphi_\alpha$. 
We have the following characterization of $\cbufR$.

\begin{proposition}\label{prop:cbuRf}
    The space $\cbufR$ has the following equivalent characterizations.
    \begin{align*}
    \cbufR&=\{a\in \bddf\mid g\mapsto \actfR{g}{a} \text{ is continuous from $\G$ to $\bddfG$} \}\\
    &=\{a\in \bddf\mid z\mapsto \actftR{z}{a} \text{ is continuous from $\B^n$ to $\bddfG$} \}\\
    &= \{a\in \bddf\mid \varphi_\alpha \Rast a \stackrel{\alpha\rightarrow\infty}{\longrightarrow} a \text{ in $\bddfG$ w.r.t. $\|\cdot\|_\infty$} \}\\
    &= \{a\in \bddf\mid a\last \varphi_\alpha\stackrel{\alpha\rightarrow\infty}{\longrightarrow} a \text{ in $\bddf$ w.r.t.  $\|\cdot\|_\infty$} \}\\
    &=\overline{\bddf\last \Lone}^{L^\infty}\\
    &=\cbufR\last \Lone.
\end{align*}
\end{proposition}

\begin{proof}
    The first equality is by definition. Now assume that $a\in \cbufR$. Then the map $\G\to \bddfG$, $g\mapsto \actfR{g}{a}$, is continuous with respect to $\|\cdot\|_\infty$ on $\bddfG$. Also, the map $\B^n\to \G$, $z\mapsto \tau_z$ is continuous. Thus, the composition $z\mapsto \actftR{z}{a}$ is continuous from $\B^n\to \bddfG$ with respect to $\|\cdot\|_\infty$. This, along with the fact that $\{\varphi_\alpha\}_{\alpha\in\N_0}$ is an approximate identity for $\Lone$, implies that $\varphi_\alpha \Rast a=a\last \varphi_\alpha \stackrel{\alpha\rightarrow\infty}{\longrightarrow} a$ in $\bddfG$ with respect to $\|\cdot\|_\infty$ (for a more detailed proof of the same nature, see the proof of Proposition~\ref{prop:convergence approx identity}). This argument proves that 
    \begin{align*}
    \cbufR
    &\subseteq \{a\in \bddf\mid z\mapsto \actftR{z}{a} \text{ is continuous from $\B^n$ to $\bddfG$} \}\\
    &\subseteq \{a \in\bddf\mid \varphi_\alpha \Rast a\stackrel{\alpha\rightarrow\infty}{\longrightarrow} \text{ in $\bddfG$ w.r.t.  $\|\cdot\|_\infty$} \}\\
    &= \{a\in \bddf\mid a\last \varphi_\alpha  \stackrel{\alpha\rightarrow\infty}{\longrightarrow} a\text{ in $\bddf$ w.r.t.  $\|\cdot\|_\infty$} \}\\
    &\subseteq \overline{\bddf\last \Lone}^{L^\infty}.
    \end{align*}

    To complete the proof, we claim that that $a\last \psi$ is right uniformly continuous for all $a\in L^\infty(\B^n)$ and $\psi\in L^1(\B^n)$. In fact,
    \[
        \actfR{g}{a\last \psi}=a\last \actfR{g}{\psi}
    \]
    for all $g\in G$. It follows that 
    \[
        \|(r(g)-r(h))(a\last\psi)\|_\infty = \|a\last(r(g)-r(h))\psi\|_\infty \leq \|a\|_\infty \|(r(g)-r(h))\psi\|_1.
    \]
    for all $g,h\in G$.  The continuity of the action of $\G$ on $\Lone$, then shows that $a\last \psi\in \cbufR$.
    Finally, since $\cbufR$ is closed in the $\|\cdot\|_\infty$ norm, we obtain that
    \[
        \overline{\bddf\last \Lone}^{L^\infty}\subseteq \cbufR,
    \]
    which, in turn, proves the equality of all of the above sets. Then the last equality in the statement is an application of the Cohen-Hewitt theorem.
\end{proof}

    \begin{lemma}\label{lem:convolution_in_cbuR}
        Let $S\in \bdd$ and let $A\in \traceop$ be a radial operator. Then $S\Last A\in \cbufR$.
    \end{lemma}

    \begin{proof}
        Since $A$ is radial, $S\Last A$ is right-$K$-invariant by Lemma~\ref{lem:associativity_radial}; that is, it is a function on $\B^n$. The proof then follows from the fact that $r(g)(S\Last A)=S\Last \actopL{g}{A}$ for all $g\in G$ and that $\G$ acts continuously on $\traceop$.
    \end{proof}


\section{Cyclic vectors and Wiener's Tauberian theorem}\label{sec:Wiener T}

In this section, we discuss Wiener's Tauberian theorem for $\Lone$ and include a proof of the fact that Toeplitz operators are dense in the space of all Schatten-$p$-class operators on $\cS^p(\cA^2)$, as was first proved in \cite{BC94}. 

\subsection{Cyclic vectors}
A function $\psi\in L^p(G)$ is said to be \textit{$p$-cyclic}  (also referred to as \textit{$p$-regular} in the literature) if the left translates of $\psi$ span a dense subspace of $L^p(G)$. A function $\psi\in \Lp$ is $p$-cyclic if the left translates of $\psi$ span a dense subspace of $\Lp$. Similarly, an operator $S\in \schop$ is $p$-cyclic ($p$-regular) if the left translates of $S$ span a dense subspace of $\schop$. Note that if $p<q$ and $S\in \schop$ is $p$-cyclic, then $S$ is also $q$-cyclic.

\begin{remark}
    While the right translations can also be used to define the notion of cyclic vectors for $\LpG$, we prefer the use of left translations. This is due to the fact that this definition does not extend to $\Lp$, since the right translations of $\psi\in\Lp$, viewed as right-$K$-invariant functions on $G$ are, in general, no longer right-$K$-invariant and thus cannot be identified with functions on $\B^n$.
\end{remark}

Quantum versions of Wiener's Tauberian theorem were studied in \cite{FG23, FLW24, LS18, W84} and a quantum Wiener's Tauberian theorem for $\LpG$, where $G$ is a locally compact group that is not necessarily abelian was discussed in  \cite{H23}. We now prove the following Wiener's Tauberian theorem for $L^p(G/K)=\Lp$. The second claim below is analogous to Theorem 7.2 in \cite{LS18}. We note that our proof does not differ significantly from the proofs of the other versions of Wiener's Tauberian theorem. In particular, we use the well-known injectivity of the Berezin transform and the map $\bddf\to \bdd, a\mapsto T_a$. However, we include the proof for completeness. Recall that $\pi$ is the discrete series representation of $\G$ on the unweighted Bergman space.

\begin{theorem}[Wiener's Tauberian theorem for $L^p(G/K)=\Lp$]\label{theo:WienerT for GK}
    Let $1\leq p<\infty$, choose $p^\prime = p/(p-1)$  so that $\frac{1}{p}+\frac{1}{p^\prime}=1$ and let $\Psi\in \schop$ be a radial operator. Then the following are equivalent:
\begin{enumerate}
        \item $\Psi$ is $p$-cyclic 
        \item $\schop=\overline{\Lone \Last \Psi}^{\cS^p}$
        \item $S\mapsto S\Last \Psi$ is injective from $\schopq{p^\prime}  \to \bddf$.
        \item $\psi \mapsto \psi\Last \Psi$ is injective from $L^{p^\prime}(\G) \to \bdd$.
        \item $L^{p^\prime} (\B^n,\lambda) =\overline{\cS^1(\cA^2)\ast_\pi \Psi}^{L^p}$.
    \end{enumerate}
\end{theorem}

\begin{proof} 
    We start by noting that  since $\Psi$ is radial then 
    $S\Last\Psi$ is right-$K$-invariant by Lemma~\ref{lem:associativity_radial}; that is, $S\Last \Psi\in \bddf$. We also recall that by Lemma~\ref{lem:ConAst} we have $L^{p'} (\G)\ast_\pi \Psi = L^{p'}(\B^n)\ast_\pi \Psi$ and by Lemma~\ref{lem:associativity_radial} the function $S\ast_\pi\Psi$ is right $K$-invariant and hence defines a function on $\B^n$.
    \medskip

    \noindent
    Note that since $\Psi$ is radial, if $S\in \schopq{p'}$ then 
    $S\Last\Psi$ is right-$K$-invariant by Lemma~\ref{lem:associativity_radial}; that is, $S\Last \Psi\in \bddf$. By Lemma~\ref{lem:duality}, we have that the map
    \begin{align*}
        \schopq{p'} \to \bddf,\quad  S\mapsto S\Last \Psi
    \end{align*}
    is the adjoint of the map
    \begin{align*}
        \Lone  \to \schop, \quad
        a \mapsto a\Last \Psi.
    \end{align*}
    Since a bounded linear map between Banach spaces is injective if and only if its adjoint has a dense image, we have proved that $(2)\Leftrightarrow (3)$. Furthermore, one has by the same lemma that the operator 
    \begin{align*}
        L^{p'}(B^n) \rightarrow \cB(\cA^2), \quad 
              a  \mapsto a\Last \Psi
    \end{align*}
    is the adjoint of the operator
    \begin{align*}
        S^1(\cA^2)\Last \Psi \rightarrow L^p(\B^n, d\lambda), \quad
                         S   \mapsto  S \Last \Psi,
    \end{align*}
    which shows that $(4)\Leftrightarrow (5)$. 

    To prove that $(3)\Rightarrow (4)$, assume that $a\in L^{p'}(\B^n, d\lambda)$ and $a\Last \Psi=0$. Then, since $\Psi$ and $\Phi$ are radial, by Lemma~\ref{lem:associativity_radial} we know that
    \[
        (a\Last \Phi)\Last \Psi=(a\Last \Psi)\Last \Phi=0.
    \]
    Hence, by (3), $a\Last \Phi=0$. Then, Lemma~\ref{lem:Toeplitz_convolution} and the injectivity of the operator
    \begin{align*}
        L^\infty(\B^n) \rightarrow \cB(\B^n), \quad
        a \mapsto T_a = a\Last\Phi,
    \end{align*}
    we have that $a=0$.

    To show that $(4)\Rightarrow(3)$, assume that $S\in \cS^{p'}(\cA^2)$ and that $S\Last \Psi=0$. Then 
    \[
        (S\Last \Psi)\Last \Phi=(S\Last \Phi)\Last \Psi=0.
    \]
    Hence $S\Last \Phi=0$ by (4). Therefore, $S=0$ by the injectivity of the Berezin transform. 

    To show that $(1)\Leftrightarrow (3)$, note that $\Psi$ is $p$-cyclic if and only if one has that for all $S\in \cS^{p'}(\cA^2)$, $S=0$ whenever $\ip{S}{\actopL{g}{\Psi}}=0$ for all $g\in \G$. Since $\ip{S}{\actopL{g}{\Psi}}=(S\Last \Psi)(g)$, the $p$-regularity of $\Psi$ is equivalent to the injectivity of $S\mapsto S\Last \Psi$.
\end{proof}
The following is proved in the same way:
\begin{theorem} 
    Let the notation be as above, let $1\leq p < \infty$, and assume that $\Psi\in \traceop$ is a radial operator. Then the following are equivalent:
    \begin{enumerate}
        \item $\Psi$ is $p$-cyclic
        \item $\schop=\overline{\Lp \Last \Psi}^{\cS^p}$
        \item $S\mapsto S\Last \Psi$ is injective from $\schopq{p'}  \to \Lq$.
        \item $a\mapsto a\Last \Psi$ is injective from $L^{p'}(G) \to \schopq{p'}$
        \item $\Lp=\overline{\schop\Last \Psi}^{L^p}$.
    \end{enumerate}
\end{theorem}
 
\begin{corollary}
    The rank-one operator $\Phi$ is $p$-cyclic for all $1\leq p <\infty$.
\end{corollary}

\begin{proof}
    By the injectivity of the Berezin transform, it follows that $\Phi$ is a $1$-cyclic operator. It is therefore $p$-cyclic for all $1\leq p < \infty$.
\end{proof}

\begin{corollary}\label{coro:regular radial function}
    Let $\Psi\in \traceop$ be a $1$-cyclic radial operator. Then
    \[
        \Lp=\overline{\Lp \Last (\Psi \Last \Psi)}
    \]
\end{corollary}
\begin{proof}
    Note that since $\Psi$ is radial, $\psi:=\Psi \Last \Psi$ is a radial function in $\Lone$. If $h\in \Lp$, then $h\Last (\Psi \Last \Psi)\in \Lp$. To prove the density, let $h\in \Lp$ and let $\epsilon>0$. Then by the second claim in Theorem~\ref{theo:WienerT for GK} there is $S\in \schop$ such that
    \[
        \|h-S\Last \Psi\|_p\leq \frac{\epsilon}{2}.
    \]
    Again by Theorem~\ref{theo:WienerT for GK}, there is $b\in \Lp$ such that $\|S-(b\Last\Psi)\|_p<\frac{(d_\pi)^{1/p}}{2\|\Psi\|_1}\epsilon$.
    Hence 
    \begin{align*}
    \|h- b\Last (\Psi\Last\Psi)\|_p &\leq \|h- S\Last\Psi\|_p+ \|(S-b\Last\Psi)\Last \Psi\|_p\\
                                    &\leq \frac{\epsilon}{2}+ \frac{1}{(d_\pi)^{1/p}}\|S-b\Last\Psi\|_p \| \Psi\|_1\\
        &< \epsilon,
    \end{align*}
    where we have used that $b\Last(\Psi\Last\Psi) = (b\Last\Psi) \Last\Psi$ by Lemma~\ref{lem:associativity}, as well as the norm estimate in~\ref{eq:op_conv_estimate_interpolation}.
\end{proof}

\subsection{Applications to Toeplitz operators}
It is already known that Toeplitz operators with $\Lp$ symbols is dense in the Schatten-$p$ class \cite{C94b}. We note that this an instance of the Wiener's Tauberian theorem as the operator $\Phi$ is $p$-cyclic. 

\begin{proposition}\label{prop:schatten_density}
    Let $1\leq p<\infty$. Then
    \begin{enumerate}
        \item $\schop=\overline{\Lone\Last \Phi}^{\cS^p}= \overline{\{T_\psi \mid \psi\in \Lone\}}^{\cS^p}$
        \item $\schop=\overline{\Lp\Last \Phi}^{\cS^p}= \overline{\{T_\psi \mid \psi\in \Lp\}}^{\cS^p}$
        \item $\Lp=\overline{\traceop\Last \Phi}^{L^p}= \overline{\{B(S) \mid S\in \traceop\}}^{L^p}$
        \item $\Lp=\overline{\schop\Last \Phi}^{L^p}= \overline{\{B(S) \mid S\in \schop\}}^{L^p}$
        \item $\Lp=\overline{\Lp\Last \varphi}^{L^p}$
    \end{enumerate}
\end{proposition}

\begin{proof}
    Recall that $T_\psi = \psi \Last\Phi$ for all $\psi\in L^p(\B_n)$, $1\leq p\leq\infty$, and also that $B(S) = S\Last\Phi$ for all $S\in \cB(\Berg)$. The statements (1) through (4) then follow immediately from Theorem~\ref{theo:WienerT for GK} and statement (5) follows from Corollary \ref{coro:regular radial function} and the fact that $\Phi\Last \Phi=\varphi$.
\end{proof}

\begin{proposition}\label{lem:schatten_cohen-Hewitt}
     Let $1\leq p < \infty$. We have the following factorization of $\schop$.
     \[
         \schop=\LoneG \Last \schop.
     \]
 \end{proposition}

 \begin{proof}
     Note that 
     \begin{align*}
         \schop & \supset \LoneG \Last \schop\\
         &\supset \LoneG \Last \Phi\\
         &\supset \Lone \Last \Phi.
     \end{align*}
     Also $\Lone \Last \Phi$ is dense in $\schop$ by Proposition~\ref{prop:schatten_density} and the fact that $\Phi$ is $p$-cyclic. Thus
     $\schop=\overline{\LoneG \Last \schop}$. Now we take $X=\schop$, $A=\LoneG$ and define the representation $T:\Lone \rightarrow\cB(\schop)$ by $T(\psi)S=\psi \Last S$ for all $S\in \schop$ and $\psi \in \Lone$. By applying the Cohen-Hewitt Theorem, we get that $\LoneG \Last \schop$ is closed in the Schatten-$p$ norm, proving the lemma.
 \end{proof}

\section{The \texorpdfstring{$\alpha$}{alpha}-Berezin transform of operators}\label{sec:Berezin transform}

In this section, we establish our main results. 
It is known by a theorem of Xia \cite{X15} that Toeplitz operators are dense in the full Toeplitz algebra. Furthermore, Suarez conjectured in \cite{S04} that $T_{B_\alpha(S)}\stackrel{\alpha\rightarrow\infty}{\longrightarrow} S$ in operator norm for every operator $S$ in $\toepalg$, where $B_\alpha(S)$ is a higher-order Berezin transform he introduced for $\cA^2(\mathbb{D})$ and was later generalized to $\Berg$ in \cite{BHV14}. Here we explore the $\alpha$-Berezin transform of operators using QHA notions and discuss this question. We also note that many nice properties of the $\alpha$-Berezin transform that were proved for $n=1$ hold more generally.

We approximate radial left-$\G$-uniformly continuous operators $S \in \cbuopLrad$ by Toeplitz operators of the form $T_{\Ber_\alpha(S)}$. As a result, we identify the radial Toeplitz algebra $\toepalgrad$ with the algebra $\cbuopLrad$ of all left-$\G$-uniformly continuous operators. Also, we prove that the norm convergence $T_{\Ber_\alpha(S)}\stackrel{\alpha\rightarrow\infty}{\longrightarrow}S$ holds for a larger subset of the Toeplitz algebra that includes all Toeplitz operators. In addition, we discuss a class of Schatten-class operators, for which the convergence $T_{B_\alpha(S)}\stackrel{\alpha\rightarrow\infty}{\longrightarrow}S$ holds in the Schatten-$p$ norm.

\subsection{The operator $\Phi_\alpha$.}
\label{sec:definition_of_operator_Phi_alpha}
Recall the functions given by 
\[
    \varphi_\alpha(z)=C_\alpha(1-|z|^2)^{n+1+\alpha}, \, z\in \B^n.
\] 
First, we define an operator $\Phi_\alpha$ such that $\Phi_\alpha\Last \Phi=\varphi_\alpha$. These operators play an important role in the analysis that follows. In particular, we define our new Berezin transform of operators using them.

For all $\alpha\in \N_0$, we define the finite-rank operator $\Phi_\alpha$ on $\Berg$ by
\[
    \Phi_\alpha=C_\alpha\sum_{k=0}^\alpha (-1)^k{\alpha \choose k}\sum_{|\vm|=k} \frac{k!}{\vm!} P_{\vm},
\]
where, for each multi-index $\vm\in\N_0^n$, we define the rank-one operator $P_{\vm}=p_\vm\otimes \overline{p_\vm}$ on $\Berg$, where in turn $p_\vm(z)=z^\vm$ for all $z\in\B^n$.

In fact, this operator can be defined for all $\alpha>-1$, but the finite sum must be replaced by an infinite sum.  If $\alpha>-1$ and $\alpha\notin\N_0$, then we define
\begin{equation}
    \label{eq:Phi_alpha_definition_infinite_series}
    \Phi_\alpha=C_\alpha \sum_{k=0}^\infty (-1)^k{\alpha \choose k}\sum_{|\vm|=k} \frac{k!}{\vm!} P_{\vm},
\end{equation}
where we recall the standard definition of non-integer Binomial coefficients:
\[
    {\alpha \choose k} := \frac{\alpha (\alpha -1) \cdots (\alpha - k+1)}{k!}.
\]
Note that if $\alpha\in\N_0$, then ${\alpha\choose k} = 0$ whenever $k > \alpha$.  We can thus use the formula (\ref{eq:Phi_alpha_definition_infinite_series}) to define $\Phi_\alpha$ for all $\alpha>-1$. 

Note that for non-integer values of $\alpha>-1$,
\[
    {\alpha \choose k} = \frac{\Gamma(\alpha + 1)}{\Gamma(\alpha - k+1)\Gamma(k+1)} = \frac{\alpha+1}{B(\alpha -k+1, k+1)},
\]
where we recall that the Beta function satisfies $B(\alpha,\beta) := \frac{\Gamma(\alpha + \beta)}{\Gamma(\alpha)\Gamma(\beta)}$.  This can also be used when $\alpha\in\N_0$, as long as $k\leq \alpha$; if $\alpha\in\N_0$ and $k>\alpha$, then we must set ${\alpha \choose k} = 0$.

Due to the fact that the rank-one operators in $\{P_\vm\}_{m\in\N_0^n}$ are scalar multiples of commuting orthogonal projections, and also to the fact that the norms of the terms of the series that defines it converge to zero, it is evident that $\Phi_\alpha$ is a compact operator on $\Berg$ and that the series converges in the operator norm.  In fact, it turns out that $\Phi_\alpha$ is a trace-class operator, as shown by the following lemma, and thus the series that defines it converges in trace norm:
\begin{lemma}\label{lem:varphi_alpha_trace_one}
    For all $\alpha>-1$, we have that $\Phi_\alpha \in \cS^1(\Berg)$, and that $\Tr \Phi_\alpha = 1$.
\end{lemma}
\begin{proof}
    We begin by noting that when $k$ surpasses $\alpha$, all further values of $(-1)^k {\alpha\choose k}$ have the same sign.  In particular, for all $k\in\N$, $k > \alpha$, we have that:
    \[
        (-1)^k {\alpha \choose k} = \left\{  
                \begin{matrix}
                    |{\alpha \choose k}|, \text{ if } \lfloor \alpha \rfloor \text{ is even}\\
                    -|{\alpha\choose k}|, \text{ if } \lfloor \alpha \rfloor \text{ is odd}\\
                \end{matrix} \right.
    \]
    
    We must first show that $\Phi_\alpha$ is a trace-class operator.  This can be done by showing that $\Tr |\Phi_\alpha| < \infty$.  

    We will do this in a slightly roundabout way, by calculating the value that $\Tr \Phi_\alpha$ must have if it is trace class:
    \begin{align*}
        \Tr \Phi_\alpha&=  C_\alpha\sum_{k=0}^\infty (-1)^k{\alpha \choose k}\sum_{|\vm|=k} \frac{k!}{\vm!} \ip{p_\vm}{p_\vm} \pagebreak[0] \\
        &= C_\alpha\sum_{k=0}^\infty (-1)^k{\alpha \choose k}\sum_{|\vm|=k} \frac{k!}{\vm!} \int_{\B^n} |z^\vm|^2 \ dz \pagebreak[0]\\
        &= C_\alpha\sum_{k=0}^\infty (-1)^k{\alpha \choose k} \int_{\B^n} |z^2|^k \ dz \pagebreak[0] \\
        &=C_\alpha  \int_{\B^n} (1-|z|^2)^{\alpha} \ dz \pagebreak[0] \\
        &= 1.
    \end{align*}
    If $\alpha\in\N_0$, then we have in fact finished the proof, because $\Phi_\alpha$ is a finite-rank operator and is therefore trace class.  Otherwise, we note that the inner sum is always finite.  We also note that in the tail, when $k>\alpha$, the terms of the sum are either all positive or all negative, so that the monotone convergence theorem can be used to justify the movement of the infinite sum into the integral in the next-to-last step.

    We now observe that 
    \[
        |\Phi_\alpha| = C_\alpha \sum_{k=0}^\infty \left(\left|{\alpha \choose k}\right| \sum_{|\vm|=k} \frac{k!}{\vm!} P_{\vm}\right).
    \]
    Thus we have that 
    \[
        \Tr |\Phi_\alpha| 
            = C_\alpha \sum_{k=0}^\infty \left|(-1)^k {\alpha \choose k} \int_{\B_n} |z^2|^k\, dz \right|.
    \]
    However, this sum clearly converges, because its tail (for $k>\alpha$) is equal (up to a global sign change) to the tail of the above series where we calculated the tentative value of $\Tr \Phi_\alpha$.  That is,
    \[
            C_\alpha \sum_{k=\lfloor \alpha\rfloor + 1}^\infty \left|(-1)^k {\alpha \choose k} \int_{\B_n} |z^2|^k\, dz \right| 
            = C_\alpha \left|\sum_{k=\lfloor \alpha\rfloor + 1}^\infty (-1)^k {\alpha\choose k} \int_{\B_n} |z^2|^k\, dz \right|,
    \]
    because all of the terms in the sum have the same sign, and thus the series for $\Tr \Phi_\alpha$ differs from the series for $\Tr |\Phi_\alpha|$ by at most a $-1$ change of finitely many terms and a multiplication by $-1$, depending on the parity of $\lfloor \alpha \rfloor$. 
\end{proof} 

\begin{lemma}\label{lem:varphi_alpha}
    For all $\alpha>-1$, 
    \begin{enumerate}
        \item For all $\alpha>-1$, $\Phi_\alpha\Last \Phi=\Phi\Last \Phi_\alpha=\varphi_\alpha$
        \item If $\alpha\in\N_0$, then $\|\Phi_\alpha\|_1=n C_\alpha \int_0^1 (1+t)^\alpha(1-t)^{n-1} dt$.
    \end{enumerate}
\end{lemma}

\begin{proof}
    We first observe that 
    \begin{align*}
        \Phi_\alpha * \Phi = \sum_{k=0}^\infty (-1)^k {\alpha \choose k} \sum_{|\vm|=k} \frac{k!}{\vm!} P_{\vm}*\Phi,
    \end{align*}
    where the sum converges in the $L^\infty$-norm for functions on $\B^n$.  This occurs because the series defining $\Phi_\alpha$ converges in the trace norm, which allows us to apply Lemma~\ref{lem:trace_class_bounded_operator_convolution_inequality}.

    We now note that, for all $z\in \C^n$,
    \[
        B(P_\vm)(z)=\ip{P_\vm k_z}{k_z}=|\ip{k_z}{p_\vm}|^2=(1-|z|^2)^{n+1}|z^{\vm}|^2.
    \]
    (recall that $k_z$ refers to the \textit{normalized} reproducing kernel).
    Therefore,
    \begin{align*}
        (\Phi_\alpha\Last \Phi)(z)&=B(\Phi_\alpha)(z) \displaybreak[0]\\
                                  &= C_\alpha\sum_{k=0}^\infty (-1)^k{\alpha \choose k}\sum_{|\vm|=k} \frac{k!}{\vm!}\ip{P_\vm k_z}{k_z} \displaybreak[0]\\
                                  &= C_\alpha(1-|z|^2)^{n+1} \sum_{k=0}^\infty (-1)^k{\alpha \choose k}\sum_{|\vm|=k} \frac{k!}{\vm!} |z^{\vm}|^2 \displaybreak[0] \\
                                  &=C_\alpha(1-|z|^2)^{n+1} \sum_{k=0}^\infty (-1)^k{\alpha \choose k} |z|^{2k} \displaybreak[0]\\
        &= C_\alpha(1-|z|^2)^{n+1+\alpha}.
    \end{align*}
    Since the Berezin transform $B(\Phi_\alpha)=\varphi_\alpha$ is a radial function, $\Phi_\alpha$ is a radial operator. Thus $\Phi_\alpha\Last \Phi=\Phi\Last \Phi_\alpha$.
    
    To prove (2), note that the eigenvalues of the radial operator $\Phi_\alpha$ are given for $k=|\vm|\leq \alpha$ by
    \begin{align*}
        \ip{\Phi_\alpha e_\vm}{e_\vm}&= C_\alpha (-1)^{k} {\alpha \choose k}\frac{k!}{\vm!}\frac{n!\vm!}{(n+k)!}\\
        &= C_\alpha(-1)^{k} {\alpha \choose k}\frac{k!n!}{(n+k)!}\\
        &= n C_\alpha(-1)^{k} {\alpha \choose k}\int_0^1 t^k(1-t)^{n-1} dt.
    \end{align*}
    Therefore 
    \begin{align*}
        \|\Phi_\alpha\|_1&=n C_\alpha\sum_{k=0}^\alpha {\alpha \choose k}\int_0^1 t^k(1-t)^{n-1} dt\\
        &=n C_\alpha \int_0^1 (1+t)^\alpha(1-t)^{n-1} dt. \qedhere
    \end{align*}
\end{proof}

\subsection{The $\alpha$-Berezin transform}
\label{sec:definition_of_new_alpha_Berezin_transform}
Recall that by Lemma~\ref{lem:Berezin of a function}, the $\alpha$-Berezin transform of a bounded function $a\in \bddf$ is given by $B_\alpha(a)=a\Last \varphi_\alpha.$  Also, the Berezin transform of $S\in \bdd$ is given by $B(S)=S\Last \Phi$. In addition, $\Phi_\alpha\Last \Phi=\varphi_\alpha$. Thus, we claim that the following definition of $\alpha$-Berezin transform of an operator is natural. Then in Lemma \ref{lem:Berezin_formula}, we prove that this definition coincides with the existing notion of the $\alpha$-Berezin transform discussed in \cite{BHV14, S04}.

\begin{definition}
    Let $S\in \bdd$ and $\alpha> -1$. We define the \textit{$\alpha$-Berezin transform} $\Ber_\alpha(S)$ of $S$ by
     \[
         \Ber_\alpha(S)=S\Last \Phi_\alpha.
     \]
\end{definition}
Note that $\Ber_\alpha(S)$ is a right-$K$-invariant function on $G$, since $\Psi_\alpha$ is a radial operator.  Thus, we identify $\Ber_\alpha(S)$ with a function on $\B^n$ and note that
    $\|\Ber_\alpha(S)\|_\infty\leq \|\Phi_\alpha\|_1\|S\|_\infty$.
In particular, $\Ber_\alpha(S)\in \cbufR$ by Lemma~\ref{lem:convolution_in_cbuR}. 

\begin{lemma}\label{lem:Berezin_formula}
    Let $\alpha>-1$ and let $S\in \bdd$. Then $\Ber_\alpha(S)$ is given by the following formula.  
    \[
        \Ber_\alpha(S)(z)=C_\alpha \sum_{k=0}^\infty(-1)^k {\alpha \choose k} \sum_{|\vm|=k} \frac{k!}{\vm!} \ip{S((p_\vm \circ \Act{z}  )k_z)}{(p_\vm \circ \Act{z})k_z}.
    \]
\end{lemma}

\begin{proof}
    Note that, by the linearity of convolutions, together with Lemma~\ref{lem:trace_class_bounded_operator_convolution_inequality}, we have that:
    \begin{align*}
        \Ber_\alpha(S)&=(S\Last \Phi_\alpha)\\
        &=C_\alpha \sum_{k=0}^\infty(-1)^k {\alpha \choose k} \sum_{|\vm|=k} \frac{k!}{\vm!} (S\Last P_\vm),
    \end{align*}
    where the sum converges uniformly on $\B_n$.
    Now note that since $P_\vm$ is radial,
    \begin{align*}
        (S\Last P_\vm)(g)&= \Tr(S\actopL{g\cdot 0}{P_m})\\
                         &=\Tr \left(\tens{(S\pi(g\cdot 0)p_{\vm})}{\pi(g\cdot 0)p_\vm}\right)\\
        &=\ip{S\pi(g\cdot 0)p_\vm}{\pi(g\cdot 0)p_\vm}\\
        &=\ip{S(p_\vm \circ \Act{g\cdot 0}  )k_{g\cdot 0}}{(p_\vm \circ \Act{g\cdot 0})k_{g\cdot 0}}.
    \end{align*}
    Hence the formula holds.
\end{proof}

\begin{lemma}
    \label{lem:Berezin_commutativity}
    Let $a\in \bddf$, $S\in \bdd$ and let $\alpha,\beta>-1$. Then
    \begin{enumerate}
        \item $\Ber_\alpha(T_a)=B_\alpha(a)$
        \item $B_\alpha(\Ber_\beta (S))=B_\beta(\Ber_\alpha (S))$
    \end{enumerate}
    
\end{lemma}

\begin{proof}
    Note that
    \begin{align*}
        \Ber_\alpha(T_a) & = \Ber_\alpha( a\Last \Phi ) \\
                        & = (a\Last\Phi) \Last \Phi_\alpha \\
                        & = a \Last (\Phi \Last \Phi_\alpha) \\
                        & = a \Last \varphi_\alpha \\
                        & = B_\alpha(a)
    \end{align*}
    for all $z\in\B^n$. To prove (2), note that 
    \begin{align*}
        B_\alpha(\Ber_\beta(S))&=(S \Last \Phi_\beta ) \Last \varphi_\alpha\\
        &= (S \Last \Phi_\beta ) \Last (\Phi_\alpha \Last \Phi)\\
        &= ((S \Last \Phi_\beta )\Last \Phi_\alpha)\Last \Phi\\
        &= ((S \Last \Phi_\alpha )\Last \Phi_\beta)\Last \Phi\\
        &=B_\beta(\Ber_\alpha(S)).
    \end{align*}
    by Lemma~\ref{lem:associativity_radial}.
\end{proof}

\subsection{A Toeplitz Criterion}

Toeplitz operators are closely related to moment sequences. For example, the eigenvalue sequence of a radial Toeplitz operator is a normalized moment sequence. The following proposition is somewhat akin to the moment sequence condition given by auxiliary sequences for the Hausdorff moment problem. This was proven for the case $n=1$ and $p=\infty$ by Suarez in \cite[Thm.\ 2.7]{S05}, and the proof, which we present below, follows a similar idea.

\begin{theorem}\label{theo:Toeplitz criterion}
    Let $1<p\leq \infty$ and let $S\in \schop$. Then the following are equivalent:
    \begin{enumerate}
        \item $S=T_a$ for some $a\in \Lp$.
        \item There exists $C>0$ such that $\|\Ber_\alpha(S)\|_p < C$ for all $\alpha\in\N_0$.
    \end{enumerate}
\end{theorem}
\begin{proof}
   To see that $(1)\implies (2)$, let $S=T_a$, where $a\in \Lp$. Then
    \[
        \|\Ber_\alpha(T_a)\|_\infty
            =\|a\Last \varphi_\alpha\|_p\leq \|a\|_\infty \|\varphi_\alpha\|_1
            =\|a\|_p.
    \]
    for all $\alpha\in \N$.

    To prove the converse, assume that there is $C>0$ such that $\|\Ber_\alpha(S)\|_p\leq C$ for all $\alpha\in \N_0$. Let $p,\in [1,\infty)$ be s.t. $\frac{1}{p}+\frac{1}{p'}=1$. Then $\Lq$ is separable and its dual is $\Lp$.  Then by the Banach-Alaoglu theorem, there exists a subsequence $\{\Ber_{\alpha_k}(S)\}_{k\in\N}$ converging to some $a\in \Lp$ in the weak-$*$ topology of $\Lp$. Then for all $f\in \Lq$,
    \[
        \lim_{k} \ip{\Ber_{\alpha_k}(S)}{f}=\ip{a}{f},
    \]
    where above $\ip{\cdot}{\cdot}$ denotes the usual $\Lp$-$\Lq$ duality pairing. Hence, for $\psi\in \Lq$, 
    \begin{align*}
        \lim_k \ip{{\Ber_{\alpha_k}(S)}\ast \varphi}{\psi}_\tr&= \lim_k \ip{(\Ber_{\alpha_k}(S)\ast \Phi)\ast \Phi}{\psi}\\
        &=\lim_k \ip{\Ber_{\alpha_k}(S)\ast \Phi}{\psi\ast \Phi}_\tr\\
        &=\lim_k \ip{\Ber_{\alpha_k}(S)}{\psi\ast \varphi}\\
        &=\ip{a}{\psi\ast \varphi}\\
        &=\ip{a\ast\varphi}{\psi}.
    \end{align*}
    Therefore, $\lim_k {\Ber_{\alpha_k}(S)}\ast \varphi=a\ast \varphi$ in the weak$^*$-topology of $\Lp$.
    Note that
    \begin{align*}
        \Ber_\alpha(S)\ast \varphi &= (S\ast \Phi_\alpha)\ast (\Phi\ast \Phi)\\
        &=(S\ast \Phi)\ast (\Phi_\alpha\ast \Phi)\\
        &= B(S)\ast \varphi_\alpha.
    \end{align*}
    Also, $B(S)\in \Lp$ for $p<\infty$ and $B(S)\in \cbufR$ by Lemma \ref{lem:convolution_in_cbuR}, and $\{\varphi_\alpha\}$ is a right approximate identity for $\Lone$.  
    Therefore, by Proposition \ref{prop:cbuRf}, and its analog for $p<\infty$, we have that 
    $$\lim_k B(S)\ast \varphi_\alpha=B(S)$$
    in $\Lp$. Now by the uniqueness of the limit,
    $$B(S)=a\ast \varphi=B(T_a).$$
    Then by the injectivity of the Berezin transform, we have that
    $$S=T_a.$$
\end{proof}

\begin{remark}
    We remark that the above Theorem \ref{theo:Toeplitz criterion} cannot be generalized to characterize Toeplitz operators with $L^1$ symbols. This is because $\Lone$ is not the dual of a Banach space. 
\end{remark}

\subsection{Approximations by Toeplitz operators.}
In this section, we consider the convergence $T_{\Ber_\alpha(S)} \stackrel{\alpha\rightarrow\infty}{\longrightarrow}S$ for $S\in \toepalg$. This question was already answered positively for radial operators in the Toeplitz algebra w.r.t. the operator norm in \cite{S05} and \cite{BHV14b}. In addition this convergence holds for Toeplitz operators \cite{S04,S07}. We explore a larger class of operators for which this convergence holds. We also consider this convergence in Schatten-classes as well as in the SOT.

\begin{lemma}\label{lem:radial Toeplitz Berezin}
    Let $S\in \bdd$  and $\alpha>-1$. Then 
    \begin{enumerate}
        \item $T_{\Ber_\alpha(S)}= B(S)\Last \Phi_\alpha$
        \item If $S$ is radial, then $T_{\Ber_\alpha(S)}=\varphi_\alpha \Last S$
    \end{enumerate}
\end{lemma}

\begin{proof}
    This is an immediate consequence of Lemma~\ref{lem:associativity} because
    \[
        S\Last \varphi_\alpha=S\Last (\Phi_\alpha\Last \Phi)=(S\Last \Phi_\alpha)\Last \Phi= (S\Last \Phi) \Last \Phi_\alpha=B(S)\Last \Phi_\alpha,
    \]
    since $\Phi_\alpha$ and $\Phi$ are radial. If $S$ is radial then we get
    $(S\Last \Phi_\alpha)\Last \Phi=(\Phi_\alpha\Last \Phi)\Last S=\varphi_\alpha \Last S.$
\end{proof}

Recall the approximate identity $\{\varphi_\alpha\}_{\alpha\in\N_0}$ of $\Lone$. Here we show that a radial operator $S\in \bdd$ satisfying certain continuities, can be approximated by convolutions $\varphi_\alpha\Last S$.

\begin{proposition}\label{prop:convergence approx identity}
    We have the following convergence properties. Let $S\in \bdd$ be a radial operator. Then
    \begin{enumerate}
        \item $\varphi_\alpha\Last S\converges{\alpha} S$ in $\bdd$ with respect to the SOT.
        \item If $S\in \cbuopL$,  then $\varphi_\alpha\Last S\converges{\alpha} S$ with respect to the operator norm in $\bdd$.
        \item If $S\in \schop$, then $\varphi_\alpha\Last S\converges{\alpha} S$ with respect to the Schatten-$p$ norm in $\schop$.
    \end{enumerate}
\end{proposition}

\begin{proof}
    We begin by proving statement (1).   Let $f_1,f_2\in \Berg$. Then since $S$ is radial, one has that $\actoptL{g}{S}=\actopL{g\cdot 0}{S}$, and hence
    \begin{align*}
        \ip{(\varphi_\alpha \Last S-S)f_1}{f_2}    &= \int_\G  \varphi_\alpha(g\cdot 0)\ip{(\actopLnoouterparens{g\cdot 0}{S}-S)f_1}{f_2}  \haar{g}\\
        &= \int_{\B^n} \varphi_\alpha(w) \ip{(\actopLnoouterparens{w}{S}-S)f_1}{f_2}  \inv{w}.
    \end{align*}
    Since $\actopLnoouterparens{0}{S}=S$, it follows that
    \[
        |\ip{(\varphi_\alpha \Last S-S)f_1}{f_2}|\leq \int_{\B^n} \|(\actopLnoouterparens{w}{S}-\actopLnoouterparens{0}{S})f_1\|\|f_2\| \varphi_\alpha (w) \inv{w}.
    \]
    Note that the map $\G\rightarrow \bdd$, $g\mapsto \actopLnoouterparens{g}{S}$, is continuous with respect to the SOT and the map $z\mapsto \tau_z$ is continuous from $\B^n$ to $\G$. Therefore the map $\B^n\rightarrow \bdd$, $z\mapsto \actopLnoouterparens{z}{S}$, is continuous with respect to the SOT. In particular, this continuity holds at $z=0$.
    Therefore, the convergence $\varphi_\alpha \Last S \converges{\alpha} S$ in the SOT, holds by standard arguments associated with approximate identities.

    The proof of (2) is similar and uses the continuity of the map $\B^n \rightarrow \bdd$, $z\mapsto \actopLnoouterparens{z}{S}$ with respect to operator norm. 
    
    We now prove (3). Since $S\in \schop$, the function $\Ber_\alpha(S)=S\Last \Phi_\alpha \in \Lp$. Therefore $T_{\Ber_\alpha(S)}=\Ber_\alpha(S)\Last \Phi\in \schop$. Also, 
    $z\mapsto\actopLnoouterparens{z}{S}$ is continuous in the Schatten-$p$ norm. We claim that
    \begin{equation}\label{eq:tracenorm bound}
         \|\varphi_\alpha\Last S-S\|_p \leq \int_{\B^n} \varphi_\alpha(w) \|\actopLnoouterparens{w}{S}-S\|_p \inv{w}.
    \end{equation}
    To prove (\ref{eq:tracenorm bound}), let $A\in \bddLtwo$, and let $\{f_k\}_{k\in\N
    }$ be an orthonormal basis of $\LtwoG$. Then for all $A\in \bdd$,
    \begin{align*}
        |\ip{\varphi_\alpha \Last S-S}{A}_{\Tr}|&=|\Tr((\varphi_\alpha \Last S-S)A)|\\
        &=\Big|\sum_{k} \int_{\B^n} \varphi_\alpha(w) \ip{(\actopLnoouterparens{w}{S}-S)Af_k}{f_k}  \inv{w} \Big|\\
        &\leq \int_{\B^n} \varphi_\alpha(w) \sum_{k} |\ip{(\actopLnoouterparens{w}{S}-S)Af_k}{f_k}| \inv{w}\\
        &\leq \int_{\B^n} \varphi_\alpha(w) \|\actopLnoouterparens{w}{S}-S\|_1  \inv{w}  \|A\|,
    \end{align*}
    proving (\ref{eq:tracenorm bound}).
    
    Hence $\varphi_\alpha\Last S\converges{\alpha} S$ in Schatten-$p$ norm by a proof similar to the proof of (1).
\end{proof}

We recall that the second statement in the following theorem was already proved in \cite{S05} and \cite{BHV14b}.
\begin{theorem}\label{theo:convergence approx identity}
    We have the following convergence properties. Let $S\in \bdd$ be a radial operator. Then
    \begin{enumerate}
        \item $T_{\Ber_\alpha(S)}\converges{\alpha} S$ in the SOT.
        \item If $S\in \schop$, then $T_{\Ber_\alpha(S)}\converges{\alpha} S$ with respect to the Schatten-$p$ norm in $\schop$.
        \item If $S\in \cbuopL$,  then $T_{\Ber_\alpha(S)}\converges{\alpha} S$ with respect to the operator norm in $\bdd$.
    \end{enumerate}
\end{theorem}

\begin{proof}
    The result follows from  Lemma~\ref{lem:radial Toeplitz Berezin} and Proposition~\ref{prop:convergence approx identity}.
\end{proof}

In the theorem below, we discuss a larger class of operators that can be approximated by Toeplitz operators.
    
\begin{theorem}\label{theo:radial covolution convergence}
    Let $1\leq p < \infty$.  Then we have the following:
    \begin{enumerate}
        \item If $S\in \Lp\Last \traceop$ then $T_{\Ber_\alpha(S)}\converges{\alpha} S$ in Schatten-$p$ norm.
        \item If $S\in \Lone\Last \schop$ then $T_{\Ber_\alpha(S)}\converges{\alpha} S$ in Schatten-$p$ norm.
        \item If $S\in \Lone\Last \cbuopL$ then $T_{\Ber_\alpha(S)}\converges{\alpha} S$ in operator norm.
    \item If $S\in \bddf\Last \traceop$ then $T_{\Ber_\alpha(S)}\converges{\alpha} S$ in operator norm.
    \end{enumerate}
\end{theorem}

\begin{proof}
    We prove only (4), as the proofs are similar. By Lemma~\ref{lem:convolution_rad}, we have that $S=a\Last A$, for some $a\in \bddf$ and $A\in \traceop^{K}$. Then 
    \begin{align*}
        B(a\Last A)= (a\Last A) \ast \Phi= a \Last (A\ast \Phi)= a \Last B(A).
    \end{align*}
    It follows that, by Lemma~\ref{lem:radial Toeplitz Berezin} and the associativity of convolution,
    \begin{align*}
        T_{\Ber_\alpha(S)} &= B(S)\Last \Phi_\alpha \ \ \text{(by Lemma~\ref{lem:radial Toeplitz Berezin})} \\
        &= (a \Last B(A)) \Last \Phi_\alpha \\
        &= a \Last (B(A) \Last \Phi_\alpha) \\
        &= a \Last T_{\Ber_\alpha(A)}\\
        &= a \Last (\varphi_\alpha \Last A) \ \ \text{(by Lemma~\ref{lem:radial Toeplitz Berezin})}.
    \end{align*}
    Since $A\in \traceop^{K}$, $\varphi_\alpha \Last A\converges{\alpha} A$ in trace norm. Therefore  $T_{\Ber_\alpha(S)}\converges{\alpha} a \Last A=S$ in operator norm.
\end{proof}

The convergence in the following corollary is immediate by the above proposition. This was originally proved by Suarez in \cite{S07} using analytical methods. 

\begin{corollary}
    Let $a\in \bddf$. Then $T_{B_\alpha(a)}\converges{\alpha} T_a$ in operator norm. Moreover, the Toeplitz algebra is generated by Toeplitz operators with right-uniformly continuous symbols and
    $$\toepalg=\overline{\cbufR\ast \Phi}.$$
\end{corollary}

\begin{proof}
    Since $B_\alpha(a)\in \cbufR$ and that Toeplitz operators are dense in the Toeplitz algebra \cite{X15}, the statement is true.
\end{proof}

\begin{corollary}\label{coro:radial convolutions in cbu}
    Let $S\in \bddf^{K}\Last \traceop$. Then $S\in \cbuopL$. In particular, radial Toeplitz operators are left-uniformly continuous.
\end{corollary}

\begin{proof}
    By Theorem~\ref{theo:radial covolution convergence}, $T_{\Ber_\alpha(S)}\converges{\alpha} S$ in operator norm. Thus it is enough to show that $T_{\Ber_\alpha(S)}\in \cbuopL$ for all $\alpha\in\N$.
    Suppose that $S= a\Last A$, where $a\in \bddf^K$ and $A\in \traceop$. Then $S=a\ast \Rad{A}$ by Lemma~\ref{lem:convolution_rad}. Note that
    \[
        \Ber_\alpha(S)=S\Last \Phi_\alpha=a\Last (\Rad{A}\Last \Phi_\alpha)=(\Rad{A}\Last \Phi_\alpha)\ast a
    \]
    by the commutativity of the convolution of radial (that is, bi-$K$-invariant) functions) in a commutative space. Note that $\Rad{A}\Last \Phi_\alpha$ is a radial function in $\Lone$ by Proposition~\ref{prop:cbu_characterizations}. Thus $\Ber_\alpha(S)\in \cbufL$ and hence $\Ber_\alpha(S)\ast \Phi=T_{\Ber_\alpha(S)}\in \cbuopL$.
\end{proof}

\subsection{A description of the radial Toeplitz algebra}

Now we characterize the radial Toeplitz algebra $\toepalgrad$ as follows.

\begin{theorem}\label{theo: main theorem}
    We can characterize the left-$\G$-uniformly continuous operators on $\Berg$ as follows
    \begin{align*}
        \cbuopL^{K} &= \overline{\bddf^{K}\Last \traceop}^{\cB}\\
        &= \{S\in \rad\mid \varphi_\alpha \Last S \converges{\alpha} S \text{ in } \|\cdot\| \}\\
        &= \{S\in \rad\mid T_{\Ber_\alpha(S)} \converges{\alpha} S \text{ in } \|\cdot\| \}\\
        &= \overline{\{T_a\mid a\in \cbufR^{K}\}}^{\cB}\\
        &= \toepalg^{K}\\
        &= \Lone^{K}\Last \cbuopL^{K}.
    \end{align*}
\end{theorem}

\begin{proof}
     One has that $T_{\Ber_\alpha(s)}= \Ber_\alpha(S)\Last \Phi$ by Lemma~\ref{lem:Toeplitz_convolution}. Moreover, if $S\in \cbuopL$, then $T_{\Ber_\alpha(S)}=\varphi_\alpha \Last S\converges{\alpha} S$ by Theorem~\ref{prop:convergence approx identity} and Lemma~\ref{lem:radial Toeplitz Berezin}. As a consequence of the facts stated above, we have
    \begin{align*}
        \cbuopL^{K} &\supset \overline{\bddf^{K}\Last \traceop}^{\cB}\ \ \ \text{(by Corollary \ref{coro:radial convolutions in cbu})}\\
        &\supset \{S\in \rad\mid T_{\Ber_\alpha(S)} \converges{\alpha} S \text{ in } \|\cdot\| \}\\
        &= \{S\in \rad\mid\varphi_\alpha \Last S \converges{\alpha} S \text{ in } \|\cdot\| \}\\
        &\supset \cbuopL^{K},        
    \end{align*}
    and hence 
    \begin{align*}
        \cbuopL^{K}&= \overline{\bddf ^{K}\Last \traceop}^{\cB}\\
        &=\{S\in \rad\mid \varphi_\alpha \Last S \converges{\alpha} S \text{ in } \|\cdot\| \}\\
        &= \{S\in \rad\mid T_{\Ber_\alpha(S)} \converges{\alpha} S \text{ in } \|\cdot\| \}.
    \end{align*}
    Since $\Ber_\alpha(S)=S \Last \Phi_\alpha  \in \cbufR$, it follows that
    \[
        \cbuopL^{K}\subseteq\overline{\{T_a\mid a\in \cbufR^{K}\}}^{\cB}\subseteq \toepalgrad.
    \]
    Finally, since the radial Toeplitz operators $T_a=a\Last \Phi$ are in $\cbuopL^{K}$ we get
    \[
     \cbuopL^{K}=\overline{\{T_a\mid a\in \cbufR^{K}\}}^{\cB}= \toepalgrad. 
    \]
    The proof of $$\cbuopL^{K}=\overline{\Lone^{K}\Last \cbuopL^{K}}^{\cB}$$
    is similar and the last equality is an application of the Cohen-Hewitt theorem.
\end{proof}

\subsection{Discussion}
By the work of Xia, we have that Toeplitz operators are dense in the full Toeplitz algebra $\toepalg$. While we do not prove that $T_{\Ber_\alpha(S)}\stackrel{\alpha\rightarrow\infty}{\longrightarrow} S$ for all operators $S\in \toepalg$, we do have that this convergence holds for Toeplitz operators and also for operators in the larger dense subset $\bddf\Last \traceop$ of $\toepalg$. 
We point out that if, in fact, $\bddf\Last \traceop$ is closed, i.e. the Toeplitz algebra factors as $\toepalg=\bddf\Last \traceop$, then the convergence $T_{\Ber_\alpha(S)}\stackrel{\alpha\rightarrow\infty}{\longrightarrow} S$ would hold in operator norm for any operator in the Toeplitz algebra $\toepalg$.

\appendix
\section{QHA on locally compact unimodular groups}\label{sec:general QHA}

In \cite{H23}, the author gives a comprehensive introduction to QHA on locally compact topological groups, which are also implicitly assumed to be second countable.  All finite-dimensional Lie groups are locally compact and second countable, so everything in that paper is applicable here.

On the other hand, we only consider semisimple Lie groups, which are unimodular.  The general theory elaborated in \cite{H23}, in which one considers groups which are not necessarily unimodular, has several technical complications.  For instance, for a square-integrable unitary representation $(\pi,\cH)$ of a unimodular group $G$, all vectors in $\cH$ are admissible (that is, they give rise to square-integrable matrix-coefficient functions), whereas if $G$ is not unimodular, only a dense subspace of vectors in $\cH$ give rise to square-integrable matrix coefficients.  Furthermore, square-integrable representations of unimodular groups satisfy a generalization of the Schur orthogonality relations that hold for compact groups in which the dimension of the representation is replaced with its \textit{formal dimension}. 

Because of the somewhat different (and simplified) nature of QHA for unimodular groups, we collect here the results that we need, and we include proofs in some of the cases in which they are simpler than in the general (not necessarily unimodular) case.  Additionally, we investigate some convergence properties that are required later. We start with a discussion on weak integrals and then review some results on tensor products and Schatten-$p$ classes of operators. 

\subsection{Weak integrals}
In this section, we very briefly review some properties of the weak integral for vector-valued functions.  We refer to \cite{Folland,R73} for discussions on weak integrals of functions with values in a topological vector space. 

Recall that if $\cX$ is a reflexive Banach space, $(X,\mathfrak{A},\nu)$ is a measure space, and $F:X\rightarrow\cX$ is a weakly measurable function such that $\int_X \|F(x)\| d\nu(x) <\infty$, then there exists exactly one element $Y\in\cX$ such that:

\[
    \langle Y, \lambda \rangle = \int_X \langle F(x), \lambda \rangle d\nu(x)
\]
for all $\lambda\in\cX^*$, where $\langle, \rangle$ denotes the bilinear pairing between $\cX$ and its dual $\cX^*$.  We define 
\[
    \int_X F(x) d\nu(x) := Y
\]
and note that 
\begin{equation}\label{eq:ineq_weak_integral}
    \left\|\int_X F(X) d\nu(x)\right\| \leq \int_X \|F(x)\| d\nu(x)
\end{equation}

When $\cX$ is not a reflexive Banach space, one may still define a sort of weak integral in many cases.  One chooses a subspace $\cC\subseteq \cX^*$ of functionals that separate the points of $\cX$, and then one defines $\int_X F(x) d\nu(x)\in\cX$ to be the unique element of $\cX$ that satisfies
\begin{equation}\label{eq:definition_weak_integral}
    \left\langle \int_X F(x)d\nu(x), \lambda \right\rangle = \int_X \langle F(x), \lambda \rangle d\nu(x) 
\end{equation}
for all $\lambda\in\cC$, whenever such an element exists.

For instance, if $\cX = \cB(\cH)$ is the Banach space of bounded operators with respect to a Hilbert space $\cH$, then one can take $\cC$ to be the space of functionals generated by those of the form:
\begin{align*}
    \lambda_{v,w} : \cB({\cH}) & \rightarrow \C \\
                        A     & \mapsto \langle Av,w \rangle
\end{align*}
for $v,w\in\cH$.  Then, for each weakly measurable function $F:X\rightarrow \cB(\cH)$ satisfying $\int_X \|F(x)\| d\nu(x) < \infty$, one has that $\int_X F(x) d\nu(x)$ exists weakly and satisfies the inequality in (\ref{eq:ineq_weak_integral}) with respect to the operator norm on $\cB(\cH)$.  
Sometimes, the integral of a $\cB(\cH)$-valued function will satisfy (\ref{eq:definition_weak_integral}) with respect to a larger class of functionals.  We will have occasion to use this in Section~\ref{sec:general QHA}. 

As another example, suppose that $\cB = \cS^1(\cH)$, the space of trace-class operators on a Hilbert space $\cH$.  Then one has the Banach-space dualities $\cK(\cH)^* = \cS^1(\cH)$ and $\cS^1(\cH)^* = \cB(\cH)$, where $\cK(\cH)$ is the space of compact operators on $\cH$, and where the dualities are given by $(A,B) \mapsto \Tr(AB)$.  For a weakly-measurable function $F:X\rightarrow \C$ such that $\int_X \|F(x)\|d\nu(x)<\infty$, one can then define $\int_X F(x)d\nu(x)\in\cS^1(\cH)$ to be the unique element satisfying 
\[
    \Tr \left(\int_X F(x) d\nu(x) A \right) = \int_X \Tr(F(x)A) d\nu(x)
\]
for all $A\in\cK(\cH)$ (we use here that $\cK(\cH)^* = \cS^1(\cH)$.  Finally, the inequality (\ref{eq:ineq_weak_integral}) follows in this case from the fact that:
\[
    \left|\int_X \Tr(F(x)A) d\nu(x)\right| \leq \int_X |\Tr( F(x)A)| d\nu(x) 
         \leq \|A\|_\infty \int_X \|F(x)\|_1 d\nu(x)
\]
for all $A\in\cK(\cH)$, which implies that the norm of $\int_X F(x)d\nu(x)$ as a functional on $\cK(\cH)$ is bounded above by $\int_X \|F(x)\|_1 d\nu(x)$. 

\subsection{Integrated representations.}
\label{sec:integrated_rep}
One of the most important uses of vector-valued integration in representation theory is the construction of the \textit{integrated representation} corresponding to a unitary representation of a locally-compact group $G$.  In particular, whenever one has a norm-preserving representation $\pi$ of $\G$ on a reflexive Banach space $\cX$, the integrated representation is the representation 
\[
    \pi : L^1(G) \rightarrow \cB(\cX),
\]
also denoted $\pi$, of the Banach $*$-algebra $L^1(G)$ (with respect to convolution) on the Banach space $\cX$.  This representation is defined as follows. For any $F\in L^1(G)$ and $v\in \cX$, we define:
\begin{equation}\label{eq:integrated_representation_definition}
    \pi (F)v = \int_\G F(x)\pi (x)v\,d\mu (x),\quad F\in L^1 (G).
\end{equation}
Recall that this weak integral satisfies 
\begin{equation} \label{eq:ConNorm}
    \|\pi (F)v\| \le \|F\|_1\|v\| .
\end{equation}
for all $F\in L^1(G)$ and $v\in \cX$.  Whenever no confusion will arise, we will also sometimes write $F*u:=F*_\pi u := \pi (F)u$.

Furthermore, if $\cX = \cB(\cH)$ is the Banach space of bounded operators, and $\pi$ is a norm-preserving strongly continuous (with respect to the strong operator topology on $\cB(\cH)$) representation on $\cB(\cH)$, then one can use the weak operator integral defined in the previous section to define the weakly-integrated representation of $L^1(G)$ in $\cB(\cH)$ (that is, $\pi(F) \in\cB(\cB(\cH))$ for all $F\in L^1(G)$) using the same integral as in $(\ref{eq:integrated_representation_definition})$, and it also satisfies $(\ref{eq:ConNorm})$.

Of course, one must show that the integrated representation $\pi$ is in fact a representation of the Banach $*$-algebra $L^1(G)$.  This is the content of the following lemma, which is well-known consequence of Fubini's theorem and a simple change of variables. We prove only the last statement. 
\begin{lemma}\label{lem:ConvRep}
    Let $(\pi, \cB)$ be a norm-preserving representation of $\G$ on a Banach space, and let $F,F_1,F_2$ be measurable functions on $G$.  Whenever the integrals used for convolution or constructing the integrated representation exist (in particular, if $F, F_1, F_2\in L^1(G)$), one has that:
    \begin{enumerate}
        \item $\pi (F_1*_\actL F_2) = \pi (F_1)\pi (F_2)$. 
        \item If $\cB$ is a Hilbert space then $\pi (F)^* = \pi (F^*)$ where $F^*(x) = \overline{F (x^{-1})}$.
        \item $\pi (x)\pi(F)u= \pi (\ell_xF)u$.
    \end{enumerate}
\end{lemma}
\begin{proof}[Proof of (3).] We have that
\[
    \pi (x)\pi (F)u = \int_\G  F(y)\pi (xy)u\haar{y} = \int_\G (F(x^{-1}y)\pi (y)u \haar{y} = \pi(\ell_x F)u.\qedhere
\]
\end{proof}
\begin{remark} 
    \textrm{Note that in Proposition~\ref{prop:trace_convolutions} we have a situation where the function is assumed essentially bounded and not integrable, but the convolution still exists.}
\end{remark} 

\subsection{Tensor products and the Schatten $p$ class}
\label{sec:tensor_products_and_Schatten_p_class}
Assume that $G$ is a locally-compact unimodular group and that $\pi$ is an irreducible unitary representation of $\G$ acting on a Hilbert space $\cH$. Let $\bddH$ be the space of all bounded operators on $\cH$ and let $\|\cdot \|$ denote the operator norm.  Assume that $\pi$ is a unitary representation of $\G$ acting on $\cH$. Then $\G$ acts on $\bddH$ by 
\[
    \tpi (x)S = \pi (x)S\pi(x)^{-1}.
\]
for all $x\in G$ and $S \in \bddH$.We note that an operator $S$ is $\G$-invariant if and only if $S$ is an intertwining operator. In particular, if $\pi$ is irreducible then $S$ is $G$-invariant if and only if $S =\lambda\rI$ for some $\lambda \in \C$.
 
 
This action by conjugation on the space $\bddH$ of bounded operators on $\cH$ can be considered as a sort of generalized tensor product representation.  For each $u,v\in \cH$, one can define the rank-one operator $\tens{u}{v}$ by 
\((\tens{u}{v})(w) := \ip{w}{v}u.\) In this fashion, one can naturally identify the space of finite-rank operators on $\cH$ with the algebraic tensor product $\cH\otimes_{\textrm{alg}}\overline{\cH}$. Similarly, the Hilbert-space tensor product $\tens{\cH}{\cH}$ (defined by taking the Hilbert-space completion of the algebraic tensor product with respect to a natural inner product one defines on it) corresponds to the space of Hilbert-Schmidt operators on $\cH$.  It turns out that the action of $\tpi$ on the Hilbert-Schmidt operators $\cS^2(\cH)$ coincides with the unitary tensor-product representation $\tens{\pi}{\pi}$, as one can readily check for rank-one operators; in fact: 
\[
    \tpi (x)(\tens{u}{v})(w) = \ip{\pi (x)^{-1}w}{v}\pi (x)u = \ip{w}{\pi (x)v}\pi (x)u = (\tens{\pi}{\pi})(x)(\tens{u}{v})(w) 
\]
for all $x\in G$ and $u,v,w\in \cH$.

 For more information on tensor products of topological vector spaces, we refer to \cite[Part III]{T67} and  \cite[Appendix 2.1]{W84}. 

We also need to introduce the spaces of \textit{Schatten-$p$-class} operators $\cS^p(\cH)$, $1\le p\le \infty$. We remind the reader that when $1\leq p < \infty$, these are the operators with finite $p$-norm  $||A||_p:=\tr(|A|^p)^{1/p}$, where $|A|=\sqrt{A^*A}$). We define $\cS^\infty(\cH) := \cB(\cH)$ to be the space of all bounded operators with the operator norm.  Note that $\cS^1(\cH)$ is the algebra of trace-class operators and $\|\cdot\|_1$ is the trace norm, while $\cS^2(\cH)$ is the space of Hilbert-Schmidt operators and that $\|\cdot\|_2$ is the Hilbert-Schmidt norm. The algebraic tensor product $\cH\otimes_{\textrm{alg}}\overline{\cH}$ is contained in $\cS^p(\cH)$ for all $p$. Futhermore, $\cS^p(\cH)$ is $\G\times\G$-invariant. If $1\le p<\infty$, then the dual of $\cS^p(\cH)$ is $\cS^{p^\prime}(\cH)$, where $p^\prime =p/(1-p)$, and where the duality is given by 
\begin{equation}\label{eq;duality}
    \ip{S}{T}_{\Tr}= \Tr ST.
\end{equation}
 
 \begin{remark} 
     In the case of the Bergman spaces one can view $\Berg\times \overline{\Berg}$ as the Bergman space of holmorphic functions $f(z,w)$ on $\B^n\times \overline{\B^n}$, where the bar indicates that we have put an opposite complex structure on the second copy of the ball. In this case, the representation of $\G$ given by the restriction to the diagonal is isomorphic to the action of $G$ on $L^2(\B^n)$; see \cite{O00} for more details.
 \end{remark}

 \begin{remark}[Coefficient functions and square-integrable representations]
     Here we have viewed tensor products as operators. But it is quite common to view them also as functions. For that, one uses the coefficient functions $\pi_{u,v}$ or the wavelet transform $W_v(u)$, defined by:
     \[
        \pi_{u,v}(x) =\ip{u}{\pi (x)v} = W_v(u)(x). 
     \]
     for all $x\in G$. We also note that $W_v:\cH \to C(\G)$ is an intertwining operator.  That is, $W_v(\pi (y)u) = \ell_y W_v(u)$ (or, said another way, $\ell_y \pi_{u,v} = \pi_{\pi(y)u,v}$) for all $y\in G$.

     Recall that an irreducible unitary representation $(\pi,\cH)$ is said to be \textit{square integrable} if there exists a vector $v\in \cH\backslash \{0\}$ such that $W_v(v)\in L^2(\G)$. If $\G$ is unimodular, as in our case, we have that $W_v(v)\in L^2(\G)$ for one vector $v\in L^2(G)\backslash\{0\}$ if and only if $W_v(u)\in L^2(G)$ for all $u,v\in L^2(G)$ (see, e.g., \cite[Section IX.3]{Knapp}, \cite[Section 1.3]{Wallach1}, or \cite[Section 6.18]{FO14}. 
     If $\pi$ is square integrable, then there exists a constant $d_\pi > 0$ such that 
     \[
         \langle \pi_{v_1,w_1}, \pi_{v_2,w_2} \rangle = \frac{1}{d_\pi} \langle v_1, v_2 \rangle \overline{\langle w_1, w_2 \rangle}
     \]
     for all $v_1,v_2,w_1,w_2\in\cH$.  One says that $d_\pi$ is the \textit{formal dimension} of $\pi$.

     Note that if $\{u_j\}_j$ is an orthogonal set of vectors in $\cH$, then $\bigoplus_j W_{u_j}(\cH )\subseteq L^2(\G )$ is an orthogonal direct sum. If $\G$ is unimodular as in our case and $\{e_j\}_j$ is an orthogonal basis, then $\bigoplus_j W_{u_j}(\cH )=L^2(\G)_{\pi}$ is the maximal $\G$-invariant subspace in $L^2(\G)$ containing a copy of $(\cH, \pi)$.
\end{remark} 

\subsection{Convolutions of functions and operators} 
\label{sec:convolutions_of_functions_and_operators}
From now on we will always assume that $\G$ is unimodular and that $(\pi,\cH)$ is a unitary representation of $\G$.  For a function $\psi$ on $\G$ and $S\in \cH$, the convolution $\psi\Gast S$ is, as always, the integrated representation:  
\begin{equation}
   \psi\Gast S := \tpi (\psi) S =\int_{\G} \psi(g) \tpi (g)S \haar{g}.
\end{equation}
In particular,
\begin{equation}\label{eq:Conv2} 
    \ip{(\psi *_\pi S) u}{v} = \int_G \psi^\vee (x) \ip{S\pi (x)u}{\pi (x) v}\ d\mu_G(x) .
\end{equation}

\begin{example} 
    Assume that $S= \tens{w}{z}$ is a rank-one operator. Then
    \[\ip{S\pi (x^{-1})u}{\pi (x^{-1})v} = \ip{u}{\pi (x)z}\ip{\pi (x)w}{v} =W_z(u)(x)\overline{W_w(v)(x)}.\]
    Assume that $\pi$ is square integrable. Then  $x\mapsto \ip{\tpi (x)Su}{v}$ is integrable and for all $u,v,z,w\in \cH \backslash \{0\}$, we have that
    \[
        | \ip{\tpi (x)Su}{v}| \le \|u\|\|v\| \|z\|\|w\| |W_{z^\prime} (u^\prime )(x)| | W_{v^\prime}(w^{\prime} )(x)|,
    \]
    where we use the notation $y^\prime := y/\|y\|$ for all $y\in \cH$. We have that $|W_{z^\prime}(u^\prime)(x)|\le 1$ for all $x\in G$, and the wavelet function vanishes at infinity. It follows that for $p\ge 1$,  
    \[  
        |W_{z^\prime} (u^\prime )(x)|^p | W_{v^\prime}(w^{\prime} )(x)|^p \le |W_{z^\prime} (u^\prime )(x)| | W_{w^{\prime} }(v^\prime )(x)|.
    \]
    The right-hand side of the above inequality, considered as a function of $x\in G$, is integrable, and
    \[
        \|W_{z^\prime}(u^\prime )\overline{W_{ w^\prime}(v^\prime )}\|_p \leq \|W_{z^\prime}(u^\prime )\|_2^{1/p} \|W_{ w^\prime}(v^\prime )\|_2^{1/p}\le
        \frac{1}{d_\pi^{1/p}}.
    \]
    It follows that $\ip{\tpi (\cdot ) S u}{v}\in L^p(\G)$ for all $p\ge 1$.  
    In particular, if $\psi \in L^{p^\prime}(\G)$, where $p^\prime = p/(p-1)$,   then the integral \eqref{eq:Conv2} exists and
    \[
        \left|\int_\G a(x) \ip{S\pi (x)u}{\pi (x)v} \haar{x}\right|
        \le \|\psi \|_{p^\prime} \|W_z(u)\|_2\|W_w(v)\|_2\le \frac{\|\psi \|_{p^\prime}}{d_\pi} \|u\|\|v\|\| z\| \| w\| .
    \]
    Furthermore, if $u=v$ and $z=w$, then we have
    \[
        \ip{\psi *S u}{u} =\int_\G \psi(x)|W_u(z)(x)|^2\haar{x} .
    \]

    We further note that $S^* = \tens{z}{w}$, $S^*S = \|z\|^2\|w\|^2 \tens{w'}{w'}$ so that $|S| = \|z\|\|w\| \tens{w'}{w'}$.
    In particular $\|S\|_p =\|z\|\| w\|$ independent of $p$, and
    \[
        \left| \int_\G \psi (x)\ip{S\pi (x)u}{\pi (x)v} \haar{x}\right| \le \frac{1}{d_\pi} \|\psi\|_{p^\prime} \| \|S\|_p \|u\| \|v\|
    \]
    for all $u,v\in \cH$.
\end{example}

From Lemma~\ref{lem:ConvRep}:

\begin{lemma}\label{lem:assoc_Lone}
    Let $S\in \bddH$ and let $\psi,\psi_1,\psi_2\in L^1(\G)$ or assume that all integrals exists.
    Then the following holds:
    \begin{enumerate}
        \item $\ds \tpi(g) (\psi \Gast S) = (\actfL{g}{\psi})\Gast S$.
        \item $\ds \psi_1 \Gast (\psi_2 \Gast S) = (\psi_1 \Gast \psi_2) \Gast S$. 
        \item $\ds \|\psi\Gast S\|_{\infty} \leq \|\psi\|_1\|S\|_{\infty}$.
    \end{enumerate} 
\end{lemma}

\begin{proof}
    Statements (1) and (2) are immediate consequences of Lemma~\ref{lem:ConvRep}, while (3) and (4) follow from (\ref{eq:ConNorm}) with respect to the weak integrals we defined for $\cB(\cH)$- and $\cS^1(\cH)$-valued functions, respectively.
\end{proof}


The norm bound in the following proposition is a less general version of \cite[Lem.\ 4.12]{H23}. Since we are assuming $G$ is unimodular, we include its more straightforward proof.

\begin{proposition}\label{prop:trace_convolutions}
   Suppose that $G$ is a unimodular locally-compact topological group.  Let $a\in L^\infty(G)$ and let $(\pi,\cH)$ be a square-integrable unitary representation of $G$ with formal dimension $d_\pi$. Let $S\in \cS^1(\cH)$ be a trace-class operator. Then the convolution $a\Gast S$ is well-defined and 
   \[
      \|a\Gast S\|_{\infty}\leq \frac{1}{d_\pi} \|a\|_\infty\|S\|_1.
   \]
   Furthermore,
   \[
      \one_G\Gast S = \int_G \pi(g) S \pi(g)^{-1} \, \haar{g} = \frac{1}{d_\pi} \Tr(S) \Id_\cH,
   \]
   where $\one_G:G\rightarrow\C$ is the constant function $g\mapsto 1$ and $\Id_\cH$ is the identity operator on $\cH$.
\end{proposition}

\begin{proof}
    Suppose that $S$ is a trace-class operator in $\cL(\cH)$.  Then there exists a partial isometry $U\in\cL(\cH)$ such that $S=U P$, where $P=\sum_{i=1}^\infty c_i e_i \otimes \overline{e_n}$ for some orthonormal basis $\{e_n\}_{n\in\N}$ for $\cH$ and some $c_i\geq 0$ for $i\in\N$ (note that the sum which defines $P$ converges in the operator norm topology). 
    We see that $\|S\|_\infty = \|P\|_\infty$ and that $\|S\|_1 = \|P\|_1= \sum_{i=1}^\infty c_i$.  Furthermore, for all $w_1, w_2\in \cH$, we have that:
    \begin{align*}
        |\langle (a \Gast S) w_1, w_2 \rangle |& \leq \int_G |a(g) \sum_{i=1}^\infty c_i\langle \pi(g)U ( e_i \otimes \overline{e_i} )  \pi(g^{-1}) w_1, w_2 \rangle|\haar{g} 
        \displaybreak[0] \\
        & \leq \|a\|_\infty \sum_{i=1}^\infty c_i \int_G | \langle \pi(g) U \langle  \pi(g^{-1}) w_1, e_i \rangle e_i, w_2 \rangle | \haar{g} 
        \displaybreak[0] \\
        & =  \|a\|_\infty \sum_{i=1}^\infty c_i \int_G | \langle w_2, \pi(g) U e_i\rangle \langle w_1, \pi(g)  e_i\rangle |\haar{g} 
        \displaybreak[0] \\ 
        & \leq \frac{1}{d_\pi}\|a\|_\infty \sum_{i=1}^\infty c_i \|\pi_{w_2,U e_i}\|_2 \|\pi_{w_1, e_i}\|_2 
        \displaybreak[0] \\
        & = \frac{1}{d_\pi} \|a\|_\infty \left(\sum_{i=1}^\infty c_i\|w_1\|\, \|w_2\|\,\|U e_i\|\,\| e_i\|\right) 
        \displaybreak[0] \\
      & \leq \frac{1}{d_\pi} \|a\|_\infty \left(\sum_{i=1}^\infty c_i \right)\, \|w_1\|\, \|w_2\|,
   \end{align*}  
   where we have used that each $e_i$ is a unit vector and that $U$ is a partial isometry. Since this holds for all $w_1,w_2\in\cH$, we can conclude that
   \[
      \|a\Gast S\| \leq \frac{1}{d_\pi} \|a\|_\infty \sum_{i=1}^\infty c_i = \frac{1}{d_\pi}\|a\|_\infty \|S\|_1.
   \]

   To prove the second assertion, we notice that, for all $w_1, w_2 \in\cH$:
   \begin{align*}
      \langle (\one_G \Gast S) w_1, w_2 \rangle & = \int_G  \sum_{i=1}^\infty c_i \langle \pi(g)U ( e_i \otimes \overline{e_i} )  \pi(g^{-1}) w_1, w_2 \rangle \haar{g} \displaybreak[0] \\
                                           & = \sum_{i=1}^\infty c_i \int_G  \langle \pi(g) U \langle  \pi(g^{-1}) w_1, e_i \rangle e_i, w_2 \rangle  \haar{g}\displaybreak[0] \\
                                           & =  \sum_{i=1}^\infty c_i \int_G  \langle w_1, \pi(g)  e_i\rangle \overline{\langle w_2, \pi(g) U e_i\rangle} \haar{g} \displaybreak[0] \\ 
      & = \frac{1}{d_\pi} \sum_{i=1}^\infty c_i \langle w_1, w_2 \rangle \langle Ue_i,e_i \rangle \displaybreak[0] \\ 
      & = \frac{1}{d_\pi} \langle w_1, w_2 \rangle \left(\sum_{i=1}^\infty \sum_{k=0}^\infty c_i \langle Ue_i, e_k \rangle \langle e_k, e_i \rangle \right) \displaybreak[0] \\
      & = \langle w_1, w_2 \rangle \frac{1}{d_\pi} \sum_{k=0}^\infty \sum_{i=1}^\infty c_i \langle U(e_i \otimes \overline{e_i})e_k, e_k \rangle \displaybreak[0] \\
      & = \frac{1}{d_\pi} \Tr(S) \langle w_1, w_2 \rangle,
   \end{align*}
   thus proving the assertion.
\end{proof}

We refer the reader to \cite[Chap. 2]{Z07} for an overview of interpolation theory, as we need to use some results in what follows. 

Let $S\in\cS^1(\cH) $ and $\psi \in L^p (\G)\cap L^1(\G)\cap L^\infty (\G)$. Then $F_S(\psi ) = \psi *S$ is well defined and contained in $\cS^1 (\cH) \cap \cS^\infty (\cH)=\cS^1(\cH)$  with norm estimates given by Lemma~\ref{lem:assoc_Lone} and Proposition~\ref{prop:trace_convolutions}. Similarly, we can fix $\psi \in L^1(\G)$ and consider the map $S\mapsto \psi\ast_\pi S$, which is well defined on $\cS^1(\cH)$ and $\cB (\cH)$, with norm estimates given by Lemma~\ref{lem:assoc_Lone}. By interpolation, we now have the following:

\begin{theorem}\label{thm:Intp1} 
    Let $(\pi,\cH)$ be a unitary representation of $G$.  Then
    \begin{enumerate}
        \item Suppose that $\pi$ is square integrable, and that $S\in \cS^1(\cH)$ and $\psi \in L^p(\G)$, $1\le p \le \infty$. Then $\psi \ast_\pi S$ extends to a bounded linear operator on $\cH$; in fact, $\psi \ast_\pi S\in \cS_p(\cH)$ with $p$-norm 
            \[
                \|\psi\ast_\pi S\|_p \le\left(\frac{1}{ d_\pi }\right)^{\frac{1}{p^\prime} }\|\psi\|_p\|S\|_1.
            \]
        \item Let $\psi\in L^1 (\G)$,  $S\in\cS^p(\cH)$, and let $1\le p \le \infty$. Then $\psi \ast_\pi S$ defines a bounded linear operator
            on $\cH$, $\psi \ast_\pi S\in \cS^p(\cH)$ and $ \|\psi\Gast S\|_p \leq \|\psi\|_1\|S\|_p$.
    \end{enumerate}
\end{theorem}


%
\subsection{A twisted convolution of two operators}
\label{sec:twisted_convolution_of_operators}
In this subsection we assume, except where stated, that $(\pi,\cH)$ is a unitary representation of the unimouldar group $\G$. Motivated by the $\cS^p-\cS^{p^\prime}$-duality from \eqref{eq;duality}, given by  $\ip{S}{T}_{\Tr} = \Tr ST$ for $S\in \cS^p$ and $T\in \cS^{p'}$, we define the \textit{(twisted) convolution} $S\ast A$ of two operators on $\cH$ to be the function on $\G$ given by
\[
    (S\Gast A)(x)=\Tr(S (\tpi (x){A})) = \Tr(S\pi(x)A\pi(x)^{-1}) = \ip{S}{\tpi (x)A},\quad x\in \G,
\]

whenever the operator $S\tpi (g)A$ is trace clase for almost all $g\in G$. In particular, this is well defined if $S\in \cS^p(\cH)$ and $A\in \cS^{p^\prime}(\cH)$. It follows almost immediately (from the fact that $\Tr(AB) = \Tr(BA)$ for all bounded operators $B$ and $A$ such that $AB$ is trace class) that
\begin{equation}
    \label{eq:operator_convolution_commutation}
    (S\Gast A)=(A\Gast S)^\vee  .
\end{equation}
 
\begin{remark} 
    One motivation to use the trace in the definition of convolution is that in non-commutative geometry trace functionals are used as a generalization of integrals (see \cite[Chap IX]{T03}). However, if one replaces $\cS^p$ (respectively, $\cS^{p^\prime}$), by $L^p (\G)$ (respectively, $L^{p^\prime}(\G)$) and take $\tpi$ as $\ell$, then this duality gives
    \[
        \ip{\psi}{\ell_g\phi} =\int \psi (x) \phi (g^{-1}x)\haar{x} = \psi * \phi^\vee (g) .
    \]
    In order to properly define a ``convolution of operators'', one would thus want to choose an anti-involution $A\mapsto A^\vee$ on $\cB(\cH)$ and then define $S*_\pi A(g) = \Tr (S\pi(g) A^\vee)$.  
        In fact, in QHA for the Fock space $\cF^2(\C^n)$, one normally chooses the antinvolution on $\cB(\cF^2(\C^n))$ to be of the form $A\mapsto VAV$, where $V\in\cB(\cF^2(\C^n))$ is defined by $Vf(z) = f(-z)$ for all $z\in \C^n$ and $f\in\cF^2(\C^n)$.  On the other hand, if we consider the Bergman space $\Berg$ as a space of right-$K$-invariant functions on $G$, then one might naively try to define an anti-involution $V$ by $VF(x) = F(x^{-1})$ for all $F:G\rightarrow \C$ and $x\in G$. However, in general, $VF$ is not right-$K$-invariant for all right-$K$-invariant functions $F:G\rightarrow\C$. Similarly, if one considers $\Berg$ as a space of functions on $G/K$, then one might be tempted to define $V$ by $VF(x\cdot 0) := F(\theta(x)\cdot 0)$ for all $x\in G$, where $\theta$ is the Cartan involution fixing $K$ (this would correspond to the mapping $V$ on functions from $\B^n$ to $\C$ defined by $Vf(z):= f(-z)$ for all $z\in \B^n$).  Unfortunately, this defines an involution rather than an anti-involution on convolution algebras of functions defined on $G/K$.  Thus, when discussing QHA for locally compact groups, it is not clear that there is a natural anti-involution that one can use here, and so we resort to using the twisted convolution we defined above. 
\end{remark}

The following commutativity and associativity properties between convolutions of different types holds.

\begin{lemma}\label{lem:associativity}
We have the following:
    \begin{enumerate}
        \item Let $S \in \cB(\cH)$, $A\in \cS^1(\cH)$, and $a\in \bddfG$.  Then
            \[
                (\psi\Gast S)\ast A= \psi \Gast (S\ast A)
            \]
        \item Suppose that $\pi$ is square integrable, and let $A,B\in\cS^1(\cH)$, and let $a\in L^\infty(G)$.  Then 
            \[
                (a \Gast A)\ast B= a \Gast (A\ast B)
            \]
    \end{enumerate}
\end{lemma}

\begin{proof}
    To prove (1), we note that 
    \begin{align*}
        ((\psi\Gast S)\Gast A)(g)&=\Tr((\psi\Gast S)\actopL{g}{A})\\
                                 &= \ip{(\psi\Gast S)}{\tpi (g) {A}}_{\Tr}\\
        &=\ip{\psi}{\tpi (g) {A}\ast S}_\tr\\
        &=\int_\G \psi(h) \Tr(\actopL{g}{A}\actopL{h}{S}) \haar{h}\\
        &=\int_\G \psi(h) \Tr(S\actopL{h^{-1}g}{A}) \haar{h}\\
        &=\int_\G \psi(h) (S\Gast A)(h^{-1}g) \haar{h}\\
        &= \psi\Gast (S\Gast A).
    \end{align*}
    The proof of (2) is similar. 
\end{proof}

\begin{lemma}\label{lem:trace_class_bounded_operator_convolution_inequality}
    Suppose that $1\leq p < \infty$ and $p'>1$ is s.t. $\frac{1}{p}+\frac{1}{p'}=1$.  Let $A\in \cS^p(\cH)$ and let $S\in \cS^{p'}(\cH) $. Then $A\ast_\pi S,\, S\ast_\pi A\in \bddfG$ and
    \[
        \|A\ast_\pi S\|_\infty = \|S\Gast A\|_\infty \leq \|A\|_p\|S\|_{p'}.
    \]
\end{lemma}
\begin{remark}
Note that the case $p=1$ applies to all trace class operators $A\in\cS^1(\cH)$ and all bounded operators $S\in\cB(\cH)$.
\end{remark}

\begin{proof}
    First, suppose that $p=1$. Recall the well-known fact that if $A\in \cB (\cH)$ and $S\in \cS^1(\cH) $, then $AS, SA \in \cS^1(\cH) $, and that, furthermore, $|\mathrm{Tr}(AS)| = |\mathrm{Tr}(SA)| \leq \|A\|_{\op} \|S\|_1$.  The result then follows from the fact that   $\|\pi(g) A \pi(g)^{-1}\|_{\op} = \|A\|_{\op}$.

    Similarly, if $1 < p < \infty$, then the fact that the bilinear map 
    \[
        \langle \, , \, \rangle_{\mathrm{\Tr}} : \cS^p(\cH) \times \cS^{p'}(\cH) \rightarrow \C
    \]   
    defined above using the trace induces a duality between $\cS^p$ and $\cS^{p'}$ implies that, for all $A\in \cS^p(\cH)$ and all $S\in\cS^{p'}(\cH)$, one has that $\Tr(AS) = \Tr(SA)$ and that $|\Tr(AS)| \leq \|A\|_p \|S\|_{p'}.$  

    Furthermore, we note that $\|\pi(g)A\pi(g)^{-1}\|_{p'} = \|A\|_{p'}\in \cS^{p'}$ for all $g\in G$.  In fact, one can show that $|\pi(g)A\pi(g)^{-1}| = \pi(g)|A|\pi(g)^{-1} \in \cS^{p'}(\cH)$ for all $g\in G$ using the polar decompositions of $A$ and of $\pi(g)A\pi(g)^{-1}$, together with the fact that $\pi$ is unitary.  Similarly, by considering the spectral decomposition of the self-adjoint operators $|A|$ and $|\pi(g)A\pi(g)^{-1}|$, one has that $|\pi(g)A\pi(g)^{-1}|^{p'} = \pi(g)|A|^{p'}\pi(g)^{-1}$.  Since $\|A\|_{p'} := \Tr(|A|^{p'})^{1/p'}$, it follows immediately that $\|\pi(g)A\pi(g)^{-1}\|_{p'} = \|A\|_{p'}$. 

    Putting everything together, we have that:
    \[
        |(S*A)(g)| = |\Tr(S \pi(g) A \pi(g)^{-1})| \leq \|S\|_p \|\pi(g)A\pi(g)^{-1}\|_{p'} =  \|S\|_p \|A\|_{p'} \\
    \]
    for all $g\in G$, as we wished to show.
\end{proof} 

\begin{lemma}[{\cite[Theorem 4.2]{H23}}] \label{lem:operator_convolutions}
    Suppose that $(\pi,\cH)$ is square integrable, and let $S,A \in \cS^1(\cH)$. Then $A\ast_\pi S,\, S\ast_\pi A\in \LoneG$ and 
    \[
        \|A\ast_\pi S\|_1 = \|S\ast_\pi A\|_{1}\leq \frac{1}{d_\pi}\|S\|_1 \|A\|_1.
    \]
    Moreover, 
    \[
        \int_\G A\ast_\pi S(g)\haar{g} = \int_G S\ast_\pi A (g) \haar{g} 
                                       = \Tr(S) \Tr(A).
    \]
\end{lemma}

\begin{theorem}[{\cite[Proposition 4.13]{H23}}] \label{thm:Intp3} 
    Assume that $(\pi,\cH)$ is square-integrable and that $1\leq p\leq \infty$. Let $S\in \cS^p(\cH)$, and $A\in\cS^1(\cH)$. Then
    \begin{equation}
        \label{eq:op_conv_estimate_interpolation}
        \|S\Gast A\|_p \leq \Big(\frac{1}{d_\pi}\Big)^{\frac{1}{p}}\|S\|_p\|A\|_1
    \end{equation}
\end{theorem}
\begin{proof} 
    This follows from Lemma~\ref{lem:operator_convolutions} by interpolation. 
\end{proof}

As a consequence of the previous theorem, we can use the Banach-space duality $\cS^1(\cH)^* = \cB(\cH)$ given by:
\begin{align*}
    \cS^1(\cH) \times \cB(\cH) & \rightarrow \C \\
    (A,B)                  & \mapsto \langle A, B \rangle_{\text{tr}} := \Tr AB
\end{align*}
to construct the convolutions $\psi *_\pi A\in\cB$ for all $A\in\cS^p$ and all $\psi\in L^{p'}$, where $1\leq p <\infty$:

\begin{lemma}
    Let $1\leq p <\infty$, let $S\in\cS^p(\cH)$, and let $\psi\in L^{p'}(G)$.  Then the weak integral
    \[
       \psi *_\pi S := \int_G \psi(x) \widetilde{\pi}(x)S\, d\mu_G(x) \in\cB(\cH)
    \]
    exists in the sense that
    \[
       \langle \psi *_\pi S, A \rangle_{\text{tr}} = \int_G \psi(x) \langle \widetilde{\pi}(x)S, A\rangle_{\text{tr}}\, d\mu_G(x)
    \]
    for all $A\in\cS^1(\cH)$.  In fact, one has that:
    \[
       \langle \psi*_\pi S, A\rangle_{\text{tr}} = \langle \psi , A*_\pi S \rangle,
    \]
    where the last pairing is the usual bilinear pairing between $L^{p'}(G)$ and $L^p(G)$.

    Furthermore, we have the norm estimate
    \[
       \|\psi*_\pi S\|_{\infty} \leq d_\pi^{-1/p} \|\psi\|_{p'} \|S\|_p 
    \]
\end{lemma}

\begin{proof}
    Fix $1\leq p < \infty$, $S\in\cS^p(\cH)$, and $\psi\in L^{p'}(\cH)$. We define the following functional on $\cS^1(\cH)$:
    \begin{align*}
        \lambda : \cS^1(\cH) & \rightarrow \C \\
        A    & \mapsto \int_G \psi(x) (A *_\pi S)(x)\, d\mu_G(x), 
    \end{align*}
    where the integral converges since $A *_\pi S\in L^p(G)$ by Theorem~\ref{thm:Intp3}.  

    Under the identification $\cS^1(\cH)^* = \cB(\cH)$, we can consider $\lambda$ to be an operator in $\cB(\cH)$, which we denote $\psi*_\pi S$.  We immediately have, for all $\psi\in L^{p'}$, $S\in \cS^p$, and $A\in\cS^1(\cH)$, that:
    \begin{align*}
        \langle \psi*_\pi S, A \rangle_{\text{tr}} & = \int_G \psi(x) (A * S)(x)\,d\mu_G(x)\\
                                                   & = \int_G \psi(x) \langle \widetilde{\pi}(x)S, A \rangle_{\text{tr}}d\mu_G(x),
    \end{align*}
    from which we conclude that 
    \[
        \psi*_\pi S = \int_G \psi(x) \widetilde{p}(x)S \in\cB(\cH),
    \]
    where the integral is taken as a weak integral with respect to the duality $\cS^1(\cH)^*= \cB(\cH)$. 

    Finally, using once more Theorem~\ref{thm:Intp3}, we observe that
    \[
        |\langle (\psi*_\pi S), A \rangle_{\text{tr}}| \leq d_\pi^{-1/p}\|\psi\|_{p'} \|S\|_p \|A\|_1,
    \]
    for all $A\in\cS^1(\cH)$, from which we conclude that $\|\psi*_\pi S\|_{\infty} \leq d_\pi^{-1/p} \|\psi\|_{p'} \|S\|_p$.
\end{proof}

\subsection{Convergence properties}
In this section we discuss convergence properties of the convolution of functions and operators as well as the convolution of two operators. The main results are collected in the following theorem: 
\begin{theorem}\label{lem:convergence}
    Let $(\pi,\cH)$ be a unitary representation of a unimodular group $G$.  Let $S\in \cB(\cH)$, and let $\{S_k\}_{k\in\N}$ be a sequence of bounded operators converging to $S$ in the strong operator topology (SOT). Then the following holds for $\psi \in L^1(\G)$ and $A,B\in \cS^1(\G)$:
    \begin{enumerate}
        \item $\psi\Gast S_k \stackrel{k\rightarrow \infty}{\longrightarrow} \psi\Gast S$.
        \item $S_k\Gast A \stackrel{k\rightarrow \infty}{\longrightarrow}  S\Gast A$ pointwise.  
        \item Suppose that $\pi$ is square integrable.  Then 
            \[
                (S_k\Gast A)\Gast B\stackrel{k\rightarrow \infty}{\longrightarrow}  (S\Gast A)\Gast B
            \]
            in the weak operator topology (WOT). 
    \end{enumerate}
\end{theorem}

\begin{proof}
    By the Uniform Boundedness Principle, there exists $C>0$ such that $\|S_k\|\leq C$ for all $k\in\N$.  Furthermore, by linearity, it is sufficient to prove the theorem for the case in which $S=0$.

    To prove (1), we fix $v\in\cH$.  It follows that, for all $k\in\N$,
    \begin{align*}
        \|(\psi \Gast S_k) v \| & = \left\|\int_G \psi(g) \pi(g) S_k \pi(g)^{-1} v \haar{g}\right\|\\
        & \leq \int_\G |\psi(g)| \| S_k \pi(g)^{-1} v\|  \haar{g}
    \end{align*}
    As long as we can find an integrable dominating function for the above integrand with is independent of $k\in\N$, the Lebesgue Dominated Convergence Theorem will imply that $\|(\psi\ast S_k)v  \| \rightarrow 0$.  In fact, $C\|v\| \psi$ is one such dominating function, since 
    \[
       |\psi(g)| \|S_k \pi(g)^{-1} v\| \leq |\psi(g)| \|S_k\|_{\op} \|v\| \leq C \|v\| \psi(g)
    \]
    for all $g\in G$ and $k\in\N$.

    \medskip

    \noindent
    (2) 
    Since $A$ is trace class there exist $f_i,h_i\in \Berg$ such that $\sum_{i=1}^\infty\|f_i\|^2< \infty$, $\sum_{i=1}^\infty\|h_i\|^2<\infty$, $A=\sum_{i=1}^\infty f_i\otimes \Bar{h_i}$, and $\Tr A=\sum_{i=1}^\infty \ip{f_i}{h_i}$. Note that, for all $g\in G$,
    \[
       S_k\tpi (g){A}=\sum_{i=1}^\infty (S_k\pi(g)f_i)\otimes\overline{\pi(g)h_i}
    \]
    and furthermore that $\sum_{i=1}^\infty \|S_k\pi(g)f_i\|^2\leq C\sum_{i=1}^\infty \|f_i\|^2 <\infty$ and $\sum_{i=1}^\infty\|\pi(g)h_i\|^2=\sum_{i=1}^\infty\|h_i\|^2<\infty$. Hence 
    \[
       (S_k\Gast A)(g)=\Tr(S_k\tpi (g) {A})=\sum_{i=1}^\infty\ip{S_k\pi(g)f_i}{\pi(g)h_i}.
    \]
    Note that $\lim_{k\to \infty }\ip{S_k\pi(g)f_i}{\pi(g)h_i}=\ip{S\pi(g)f_i}{\pi(g)h_i}$, since SOT convergence implies WOT convergence. Also,
    \[
       \sum_{i=1}^\infty|\ip{S_k\pi(g)f_i}{\pi(g)h_i}|\leq \sum_{i=1}^\infty \|f_i\|\|h_i\|\leq \Big(\sum_{i=1}^\infty \|f_i\|^2\Big)^{\frac{1}{2}}\Big(\sum_{i=1}^\infty \|h_i\|^2\Big)^{\frac{1}{2}}<\infty.
    \]
    Therefore, $\lim_{k\to \infty }(S_k\Gast A)(g)=(S\Gast A)(g)$ for all $g\in G$.
\medskip

\noindent
(3)  Note that $\|S_k\Gast A\|_\infty\leq C\|A\|_1$ for all $k\in\N$. Furthermore, we have that $S_k\Gast A\converges{k} S\Gast A$ pointwise. In addition, for all $v,w\in \cH$, the function $g\mapsto \ip{\tpi (g){A}v}{w}$ is in $L^1(G)$ by Theorem~\ref{prop:trace_convolutions}. The Lebesgue Dominated Convergence Theorem then implies that $\|(\psi * S_k)v\| \converges{k}0$.
\end{proof}

\subsection{Duality} 
In this section we always assume that $(\pi,\cH)$ is a square-integrable representation of a unimodular group $\G$. Recall the duality $\cS^p-\cS^{p^\prime}$-duality from \eqref{eq;duality} given by $\ip{S}{T}_{\Tr} = \Tr ST$ for $S\in \cS^p$ and $T\in \cS^{p'}$, 
%
and similarly the $L^p-L^{p^\prime}$-pairing given by $\ip{\varphi}{\psi}_{L_1}= \int_{G} \varphi(g)\psi(g) \haar{g}$ for $\varphi \in L^p$ and $\psi\in L^{p'}$.

\begin{lemma} \label{lem:duality}
    Suppose that $(\pi ,\cH)$ is a unitary representation of a unimodular group $G$, and let $\psi\in \LoneG$, $a\in \bddfG$, $S\in \cS^p(\cH)$ where $1\leq p\leq \infty$ and let $A\in \cS^1(\cH)$.  Then
    \[ \ip{\psi\Gast S}{B}_{\Tr}=\ip{\psi}{B\Gast S}_{L^1}\ \ \forall B\in \cS^{p^\prime}(\cH).\]
\end{lemma}

\begin{proof}
     By Young's inequality (Theorem~\ref{thm:Intp1}), we know that $\psi\Gast S$ is in $\cS^p(\cH)$.
    Let $B\in \cS^{p'}(\cH)$ and let $\{f_k\}$ be an orthonormal basis. Then
    \begin{align*}
        \ip{\psi \Gast S}{B}_{\Tr} &= \Tr((\psi  \Gast S)B)\\
        &= \Tr(B(\psi \Gast S))\\
        &= \sum_{k} \ip{B(\psi \Gast S)f_k}{f_k}\\
        &= \sum_{k} \int_{G} \psi(g)\ip{B\actopL{g}{S}f_k}{f_k} \haar{g}\\
        &= \int_{G} \psi(g)  \sum_{k}\ip{B\actopL{g}{S}f_k}{f_k} \haar{g}\\
        &=\int_{G} \psi(g)  \Tr(B\actopL{g}{S}) \haar{g}\\
        &=\ip{a}{B\Gast S}_{\Tr}. \qedhere
    \end{align*}
\end{proof}

\begin{remark}
The above lemma shows that, when defined on appropriate spaces, the map $S\mapsto S\ast \Psi$ and the map $a\mapsto a\ast \Psi$ are adjoints of each other, where $\Psi$ is a fixed operator.
\end{remark}

\section{Table of Notation}\label{sec:table}

\begin{tabularx}{5.7in}{|c|X|}
    \hline
    Notation & Description \\
    \hline
    $\cA^2(\B^n)$ & Bergman space $\cO(\B^n)\cap L^2(\B^n, \nu)$ of square-integrable holomorphic functions on the unit ball in $\C^n$ with respect to normalized Lebesgue measure $\nu$ (so that $\nu(\B^n) =1$).\\
    \hline
    $\cA^2_{\alpha}(\B^n)$ & Weighted Bergman space $\cO(\B^n) \cap L^2(\B^n, \mu_\alpha)$, where $d\mu_\alpha(z) = C_\alpha (1-|z|^2)^\alpha) d\nu(z)$.  Here $\alpha > -1$.\\
    \hline 
    $C_\alpha$ & Normalization factor so that $\mu_\alpha(\B^n) = 1$; $C_\alpha = \frac{(n+\alpha)!}{n!\alpha!}$. \\
    \hline
    $K^\alpha$ & Reproducing kernel for $\cA^2_\alpha(\B^n)$; see Equation~\ref{eq:reproducing_kernel_definition}. \\
    \hline
    $K^\alpha_z$ & Reproducing function for $z\in \B^n$; $K^\alpha_z = K^\alpha(\cdot, z)$; see Equation~\ref{eq:reproducing_kernel_definition}. \\
    \hline
    $k^\alpha_z$ & Normalized reproduction function; that is, $k^\alpha_z = K^\alpha_z / \|K^\alpha_z\| \in \cA^2_\alpha(\B^n)$; see Equation~\ref{eq:kz}. \\ 
    \hline
    $p_{\textbf{m}}$ & Where $\textbf{m}\in\N^n$ is a multi-index, this is the monomial function given by $p_{\textbf{m}}(z) := z_1^{m_1} \cdots z_n^{m_n}$; see Equation~~\ref{eq:onb}. \\ 
    \hline
    $e_{\textbf{m}}$ & Normalized monomial; $e_{\textbf{m}}:= p_{\textbf{m}} / \| p_{\textbf{m}}\|$; see Equation~\ref{eq:onb}.\\
    \hline 
    $\pi_\alpha$ & Holomorphic discrete series representation; see Equation~\ref{eq:representation_definition}. \\
    \hline
    $j_\alpha$ & Cocycle for the holomorphic discrete series representation $\pi_\alpha$; see Equation~\ref{eqn:cocycle_def}.\\
    \hline
    $K, K_z, k_z, \pi, j$ & The same as $K^\alpha$, $K_z^\alpha$, $k^\alpha_z$, $\pi_\alpha$, and $j_\alpha$, respectively, with $\alpha=0$ (corresponding to unweighted Bergman space $\cA(\B^n)$. \\
    \hline
    $\tau_z$   & Biholomorphic involution interchanging $z$ and $0$; see Section~\ref{sec:involutions_and_lifting}. \\
    \hline
    $\widetilde{\pi}$ & Conjugation action of $G$ on $\cB(\cA(\B^n))$ defined by $\widetilde(g)A := \pi(g)A\pi(g^{-1})$; see Section~\ref{sec:left_translations_of_operators}\\
    \hline
    $\ell_g \psi$ & Left translation of function $\psi:\B^n\rightarrow \C$ by $g\in G$; see Section~\ref{ss:Translation}.\\
    \hline
    $r_g \psi$ & Right translation of function $\psi:\B^n\rightarrow \C$ by $g\in G$; see Section~\ref{ss:Translation}.\\
    \hline
    $\psi*_\ell F$, $\psi*_r F$ & Left and right convolutions s of $\psi$ and $F$; see~Section~\ref{ss:Conv}.\\
    \hline
    $\psi *_\pi A$ & Convolution of function $\psi$ with operator $A\in \cB(\cA^2(\B^n))$defined using the integrated representation of $\widetilde{\pi}$; see Section~\ref{sec:convolutions_of_functions_and_operators}. \\
    \hline
    $S *_\pi A$    & Twisted convolution of two operators $S,A\in \cB(cA^2(\B^n))$, defined using the trace; see Section~\ref{sec:twisted_convolution_of_operators}.\\ 
    \hline
    $C^{(L)}_{b,u}(\B^n)$ & Bounded, left-$G$-uniformly continuous functions on the unit ball; see Definition~\ref{def:cbul:functions}.\\
    \hline
    $C^{(L)}_{b,u}(\cA^2)$ & Bounded, left-$G$-uniformly continuous operators in $\cB(\cA^2(\B^n))$; see Definition~\ref{def:cbul:operators}.\\
    \hline
    $C^{(R)}_{b,u}(\cA^2)$ & Bounded, right-$G$-uniformly continuous functions on the unit ball; see Definition~\ref{def:cbur:functions}.\\
    \hline
    $\varphi_\alpha$   & Defined by $\varphi_\alpha(z) = C_\alpha (1-|z|^2)^{\alpha+n+1} = C_\alpha / K^\alpha(z,z)$ for all $z\in\B^n$. Forms approximate indentity for convolution.  See Section~\ref{sec:natural_approximate_identity}\\
    \hline
    $\Phi_\alpha$  & Trace-class operator defined in Section~\ref{sec:definition_of_operator_Phi_alpha} as a sort of operator version of $\varphi_\alpha$. Used to define $\alpha$-Berezin transform of operators.\\ 
    \hline
    $B(S)$  & Berezin transform of operator $S\in\cB(\cA^2(\B^n)$; $B(S)(z):= \langle Sk_z, k_z\rangle$ for all $z\in \B^n$; see Equation~\ref{eq:BerS}.\\
    \hline 
    $B(a)$  & Berezin transform of function $a:\B^n\rightarrow \C$; see Equation~\ref{eq:BerFunction}. \\
    \hline
    $B_\alpha(a)$ & Standard $\alpha$-Berezin transform of $a:\B^n\rightarrow \C$; same as Berezin transform of $a$ in $\cA^2_\alpha(\B^n)$; see Equation~\ref{eq:BerAlphaDefinition}.\\
    \hline\\[-12pt]
    $\widetilde{B_\alpha}(S)$ & An $\alpha$-Berezin transform of the operator $S\in\cB(\cA^2(\B^n))$; see Section~\ref{sec:definition_of_new_alpha_Berezin_transform}. \\
    \hline
    $\mathfrak{T}(L^\infty)$ & Toeplitz algebra; the $C^*$-algebra generated by $T_a$ with $a\in  L^\infty(\B^n)$. See Section~\ref{sec:prelim}. \\
    \hline
    $p'$ & If $1\leq p \leq \infty$, then $p':= p/(p-1)$, so that $\frac{1}{p} + \frac{1}{p'} = 1$. \\
    \hline
    $\cS^p(\cH)$ & If $1\leq p < \infty$, then this is the Schatten $p$ class of operators on the Hilbert space $\cH$.  If $p=\infty$, then for convenience we put $\cS^\infty(\cH)=\cB(\cH)$. See Section~\ref{sec:tensor_products_and_Schatten_p_class}.\\
    \hline
\end{tabularx}

\noindent \textbf{Acknowledgement:} The authors would like to thank the referee for their helpful comments and suggestions, which helped us improve the quality of the paper.\\

\noindent \textbf{Data availability:} Data sharing is not applicable to this article. \\

\noindent \textbf{Conflict of interest:} The authors have no financial or proprietary interests in any material discussed in this article.\\

\noindent \textbf{Funding:} Mishko Mitkovski was supported by was supported by NSF grant DMS-2000236. Gestur \'Olafsson was supported by Simons grant 58606.

\bibliographystyle{plain} 
\bibliography{main}

\end{document}